\documentclass[aos,MSNbibl,nameyear,seceqn,dvips]{arximspdf}
\usepackage{multirow,dcolumn}
\usepackage{graphicx}

\doi{10.1214/14-AOS1243} 
\volume{42}
\issue{6}
\pubyear{2014}
\firstpage{2202}
\lastpage{2242}
\docsubty{FLA}

\makeatletter
\newcolumntype{d}[1]{D{.}{.}{#1}}
\newcommand{\eqref}[1]{(\ref{#1})}
\def\cal{\mathcal}

\newtheorem{thmm}{Theorem}[section] 
\newtheorem{lemma}{Lemma}[section] 

\newproclaim{definition}{Definition}[section]
\newproclaim{condition}{Condition}[section]

\newproclaim{example}{Example}
\newcommand{\eps}{\varepsilon}
\newcommand{\goto}{\rightarrow}
\newcommand{\margmin}{\operatorname{argmin}}
\newcommand{\cG}{{\cal G}}
\newcommand{\cU}{{\cal U}}
\newcommand{\hamm}{\operatorname{Hamm}}
\newcommand{\sgn}{\operatorname{sgn}}
\newcommand{\ty}{\tilde{Y}}
\newcommand{\hb}{\hat{\beta}}
\newcommand{\call}{{\cal I}}
\newcommand{\callJ}{{\cal J}}
\newcommand{\cM}{\mathcal{M}}
\newcommand{\tq}{{\cal Q}}
\newcommand{\tomega}{\tilde{\omega}}
\makeatother

\begin{document}
\begin{frontmatter}

\title{Covariate assisted screening and estimation}
\runtitle{Covariate assisted screening}

\begin{aug}
\author[A]{\fnms{Zheng Tracy}~\snm{Ke}\corref{}\thanksref{T1,T2,T3}\ead[label=e1]{zke@galton.uchicago.edu}},
\author[B]{\fnms{Jiashun}~\snm{Jin}\thanksref{T2}\ead[label=e2]{jiashun@stat.cmu.edu}}
\and
\author[C]{\fnms{Jianqing}~\snm{Fan}\thanksref{T1}\ead[label=e3]{jqfan@princeton.edu}}
\runauthor{Z.~T. Ke, J. Jin and J. Fan}
\thankstext{T1}{Supported in part by NSF Grant DMS-07-04337,
the National Institute of General Medical Sciences of the National
Institutes
of Health Grants R01GM100474 and R01-GM072611.}
\thankstext{T2}{Supported in part by NSF CAREER award DMS-09-08613.}
\thankstext{T3}{The major work of this article was completed when Z. Ke
was a
graduate student at Department of Operations Research and Financial
Engineering, Princeton University.}

\affiliation{University of Chicago, Carnegie Mellon University and
Princeton~University}

\address[A]{Z. T. Ke\\
Department of Statistics\\
University of Chicago\\
Chicago, Illinois 60637\\
USA\\
\printead{e1}}

\address[B]{J. Jin\\
Department of Statistics\\
Carnegie Mellon University\\
Pittsburgh, Pennsylvania 15213\\
USA\\
\printead{e2}}

\address[C]{J. Fan\\
Department of Operations Research\\
\quad and Financial Engineering\\
Princeton University\\
Princeton, New Jersey 08544\\
USA\\
\printead{e3}}

\end{aug}

\received{\smonth{11} \syear{2013}}
\revised{\smonth{4} \syear{2014}}

%
\begin{abstract}
Consider a linear model $Y = X \beta+ z$, where $X = X_{n,p}$ and $z
\sim N(0, I_n)$.
The vector $\beta$ is unknown but is sparse in the sense that most of
its coordinates are $0$. The main interest is to separate its nonzero
coordinates from the zero ones (i.e., variable selection). Motivated by
examples in long-memory time series
(Fan and Yao [\textit{Nonlinear Time Series: Nonparametric and Parametric Methods}
(2003) Springer]) and the
change-point problem
(Bhattacharya [In \textit{Change-Point Problems ({S}outh {H}adley, MA, 1992)}
(1994) 28--56 IMS]),
we are primarily interested in the case where the Gram matrix $G=X'X$
is \textit{nonsparse}
but \textit{sparsifiable}
by a finite order linear filter. We focus on the regime where signals
are both \textit{rare and weak} so that successful variable selection
is very challenging but is still possible.

We approach this problem by a new procedure called the \textit
{covariate assisted screening and estimation} (CASE). CASE first uses a
linear filtering to reduce the original setting to a new regression
model where the corresponding Gram (covariance) matrix is sparse.
The new covariance matrix induces a sparse graph, which guides us to
conduct multivariate screening without visiting all the submodels. By
interacting with the signal sparsity, the graph enables us to decompose
the original problem into many separated small-size subproblems (if
only we know where they are!). Linear filtering also induces a
so-called problem of \textit{information leakage}, which can be
overcome by the newly introduced \textit{patching} technique. Together,
these give rise to CASE, which is a two-stage \textit{screen and clean} [Fan and
Song \textit{Ann. Statist.} \textbf{38} (2010) 3567--3604; Wasserman
and Roeder \textit{Ann. Statist.} \textbf{37} (2009) 2178--2201]
procedure, where we first identify candidates of these submodels by
\textit{patching and screening}, and then re-examine each candidate to
remove false positives.

For any procedure $\hat{\beta}$ for variable selection, we measure the
performance by the minimax Hamming distance between the sign vectors of
$\hat{\beta}$ and $\beta$.
We show that in a broad class of situations where the Gram matrix is
nonsparse but sparsifiable, CASE achieves the optimal rate of
convergence. The results are successfully applied to long-memory time
series and the change-point model.
\end{abstract}

%
\begin{keyword}[class=AMS]
\kwd[Primary ]{62J05}
\kwd{62J07}
\kwd[; secondary ]{62C20}
\kwd{62F12}
\end{keyword}

\begin{keyword}
\kwd{Asymptotic minimaxity}
\kwd{graph of least favorables (GOLF)}
\kwd{graph of strong dependence (GOSD)}
\kwd{Hamming distance}
\kwd{multivariate screening}
\kwd{phase diagram}
\kwd{rare and weak signal model}
\kwd{sparsity}
\kwd{variable selection}
\end{keyword}
\end{frontmatter}

\section{Introduction} \label{sec:intro}
Consider a linear regression model\label{p1}
%
\begin{equation}
\label{model} Y=X\beta+ z,\qquad X = X_{n,p}, z\sim N\bigl(0,
\sigma^2I_n\bigr).
\end{equation}
The vector $\beta$ is unknown but is sparse, in the sense that only a
small fraction of its coordinates is nonzero. The goal is to separate
the nonzero coordinates of $\beta$ from the zero ones (i.e., variable
selection).
We assume $\sigma$, which is the standard deviation of the noise,
is known
and set $\sigma=1$ without loss of generality.

In this paper, we assume the Gram matrix
%
\begin{equation}
\label{GramMatrix} G = X'X
\end{equation}
is normalized so that all of the diagonals are $1$, instead of $n$ as
is often used in
the literature. The difference between two normalizations is nonessential,
but the signal vector $\beta$ are different by a factor of $\sqrt{n}$.

We are primarily interested in the cases where:
\begin{itemize}
\item the signals (nonzero coordinates of $\beta$) are rare (or
sparse) and weak;
\item the Gram matrix $G$ is \textit{nonsparse} or even ill-posed
(but it may be \textit{sparsified} by some simple operations; see
details below).
\end{itemize}
In such cases, the problem of variable selection is new and challenging.

While signal rarity is a well-accepted concept, signal weakness is an important
but a largely neglected notion, and many contemporary researches on
variable section have been focused on the regime where the signals are
\textit{rare but strong}.
However, in many scientific experiments, due to
the limitation in technology and constraints in resources,
the signals are unavoidably weak. As a result, the signals are hard to
find, and it is easy to be fooled.
Partially, this explains why many published works (at least in some
scientific areas)
are not reproducible; see, for example, \citet{Ioannidis}.

We call $G$ \textit{sparse} if each of its rows has relatively few
``large'' elements, and we call $G$ sparsifiable if $G$ can be reduced
to a sparse matrix by some simple operations (e.g., linear filtering or
low-rank matrix removal). The Gram matrix plays a critical role in
sparse inference, as the sufficient statistics $X'Y \sim N(G\beta, G)$.
Examples where $G$ is nonsparse but sparsifiable can be found in the
following application areas:
\begin{itemize}
\item\textit{Change-point problem}. Recently, driven by researches on
DNA copy number variation, this problem has received a resurgence of
interest [\citet{SaRa,Olshen04,TibWang}]. While existing literature
focuses on \textit{detecting} change-points,
\textit{locating} change-points is also of major interest in many
applications [\citet{Andr02,Siegmundperson,ZhangSieg}]. Consider a
change-point model
%
\begin{equation}
\label{changept} Y_i = \theta_i + z_i,\qquad
z_i \stackrel{\mathrm{i.i.d.}} {\sim} N(0,1), 1 \leq i \leq p,
\end{equation}
where $\theta= (\theta_1, \ldots, \theta_p)'$ is a piece-wise constant
vector with jumps at relatively few locations. Let $X = X_{p, p}$ be
the matrix such that $X(i,j) = 1\{j \geq i\}$, $1 \leq i, j \leq p$.
We re-parametrize the parameters by
\[
\theta= X \beta \qquad\mbox{where } \beta_k = \theta_{k} -
\theta _{k+1}, 1 \leq k \leq p-1, \mbox{ and } \beta_p =
\theta_p,
\]
so that $\beta_k$ is nonzero if and only if $\theta$ has a jump at
location $k$.
The Gram matrix~$G$ has elements $G(i,j) = \min\{i,j\}$, which is
evidently nonsparse. However, adjacent rows of $G$ display a high
level of similarity, and the matrix
can be sparsified by a second order adjacent differencing between the rows.
\item\textit{Long-memory time series}.
We consider using time-dependent data to build a prediction model for
variables of interest,
$Y_t= \sum_j\beta_jX_{t-j}+\varepsilon_t$,
where $\{X_t\}$ is an observed stationary time series and $\{\varepsilon
_t\}$ are white noise. In many applications, $\{X_t\}$ is a long-memory
process. Examples include volatility process [\citet{FanTimeS,RayTsay}],
exchange rates, electricity demands, and river's outflow (e.g., the
Nile's). Note that the problem can be reformulated as~\eqref{model},
where the Gram matrix $G=X'X$ is asymptotically close to the
auto-covariance matrix of $\{X_t\}$ (say, $\Omega$). It is well known
that $\Omega$ is Toeplitz, the off-diagonal decay of which is very slow
and the matrix $L^{1}$-norm which diverges as $p \goto\infty$.
However, the Gram matrix can be sparsified by a first order adjacent
differencing between the rows.
\end{itemize}
Further examples include jump detections in (logarithm) asset prices
and time series following a FARIMA model [\citet{FanTimeS}].
Still other examples include the factor models, where $G$ can be
decomposed as
the sum of a sparse matrix and a low rank (positive semi-definite)
matrix. In these examples,
$G$ is nonsparse, but it can be sparsified either by adjacent row
differencing or low-rank matrix removal.

\subsection{Nonoptimality of $L^0$-penalization method for rare and weak signals} \label{subsec:L0nonopt}
When the signals are rare and strong, the problem of variable selection
is more or less well understood. In particular, \citet{DonohoStarck}
[see also \citet{DonohoHuo}] have investigated the \textit{noiseless
case} where they reveal a fundamental phenomenon. In detail, when there
is no noise, model (\ref{model}) reduces to
$Y = X \beta$. Now, suppose $(Y, X)$ are given, and
consider the equation $Y = X \beta$. In the general
case where $p > n$, it was shown in \citet{DonohoStarck} that under
mild conditions
on $X$, while the equation $Y = X \beta$ has infinitely many solutions,
there is
\textit{a unique} solution that is \textit{very sparse}. In fact, if
$X$ is full rank and
this sparsest solution has $k$ nonzero elements,
then all other solutions have at least $(n - k + 1)$ nonzero
elements; see Figure~\ref{Fig:add} (left).

\begin{figure}
\includegraphics{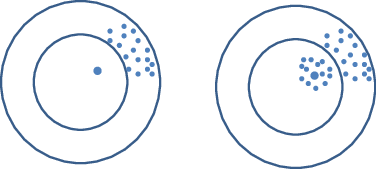}
\caption{Illustration for solutions of $Y = X \beta+ z$ in the
noiseless case (left; where $z = 0$)
and the strong noise case (right).
Each dot represents a solution (the large dot is the ground truth),
where the distance to the center is the $L^0$-norm
of the solution. In the noiseless case, we only have one very sparse
solution, with all other being much denser. In the strong noise case,
signals are rare and weak, and we have
many very sparse solutions that have comparable sparsity to that of the
ground truth.} \label{Fig:add}
\end{figure}

In the spirit of Occam's razor, we have reason to believe that
this \textit{unique sparse} solution is the ground truth we
are looking for. This motivates the well-known
method of $L^0$-penalization, which looks for the sparsest
solution where the sparsity is measured by the $L^0$-norm.
In other words, in the noiseless case, the $L^0$-penalization
method is a ``fundamentally correct'' (but computationally
intractable) method.

In the past two decades, the above observation has motivated a long
list of
\textit{computable global penalization methods}, including
but not limited to the lasso, SCAD, MC$+$,
each of which hopes to produce solutions that approximate
that of the $L^0$-penalization method.

These methods usually use a theoretic framework that
contains four intertwined components: ``signals
are rare but strong,'' ``the true $\beta$ is the sparsest
solution of $Y = X \beta$,'' ``probability of exact recovery
is an appropriate loss function'' and ``$L^0$-penalization
method is a fundamentally correct method.''

Unfortunately, the above framework is no longer appropriate
when the signals are rare and weak.
First, the fundamental phenomenon found in \citet{DonohoStarck}
is no longer true. Consider the equation $Y = X \beta+ z$, and
let $\beta_0$ be the ground truth. We can produce many vectors $\beta
$ by
perturbing $\beta_0$ such that two models $Y = X \beta+ z$ and $Y = X
\beta_0 + z$
are indistinguishable (i.e., all tests---computable or not---are
asymptotically powerless). In other words, the equation $Y = X \beta+
z$ may have many \textit{very sparse} solutions,
where the ground truth is not necessarily the sparsest one; see Figure~\ref{Fig:add} (right).

In summary, when signals are rare and weak:
\begin{itemize}
\item The situation is more complicated than that considered by
\citet{DonohoStarck}, and the principle of Occam's razor is less relevant.
\item``Exact recovery'' is usually impossible, and the Hamming distance
between the sign vectors of $\hat{\beta}$ and $\beta$ is a more
appropriate loss function.
\item The $L^0$-penalization method is not ``fundamentally correct'' if
the signals are rare/weak and the Hamming distance is the loss function.
\end{itemize}
For example, it was shown in \citet{UPS} that in the rare/weak regime,
even when $X$ is very simple and when the tuning parameter is ideally
set, the $L^0$-penalization method is not rate optimal in terms of the
Hamming distance.
See \citet{UPS} for details.

\subsection{Limitation of UPS}\label{subsec:gs}
That the $L^0$-penalization method is rate nonoptimal implies that
many other penalization methods (such as the lasso, SCAD, MC$+$) are also
rate nonoptimal in the rare/weak regime.

What could be rate optimal procedures in the rare/weak regime?
To address this, \citet{UPS} proposed a method called
\textit{univariate penalization screening (UPS}), and showed that UPS
achieves the
optimal rate of convergence in Hamming distance under certain conditions.

UPS is a two-stage screen and clean method [\citet{Wasserman}], at the heart
of which is marginal screening. The main challenge that marginal screening
faces is the so-called phenomenon of \textit{signal cancellation}, a
term coined
by \citet{Wasserman}. The success
of UPS hinges on relatively strong conditions [e.g., see \citet
{Genovese}], under which signal cancellation
has negligible effects.

\subsection{Advantages and disadvantages of sparsifying}
Motivated by the application examples aforementioned, we are interested
in the rare/weak
cases where $G$ is nonsparse but can be sparsified by a finite-order
linear filtering.
That is, if we denote the linear filtering by a $p \times p$ matrix
$D$, then the matrix
$DG$ is sparse in the sense that each row has relatively few large
entries, and all other entries
are relatively small.

In such challenging cases, we should not expect the $L^0$-penalization
method or the
UPS to be rate optimal; this motivates us to develop a new approach.

Our strategy is to use sparsifying and so to exploit the sparsity of
$DG$. Multiplying both sides of (\ref{model}) by $X'$ and then by $D$ gives
%
\begin{equation}
\label{modeladd} d = D G \beta+ N\bigl(0, D G D'\bigr),\qquad d \equiv D
\tilde{Y}, \tilde{Y} \equiv X' Y.
\end{equation}
On one hand, sparsifying is helpful for both matrices $DG$ and $DGD'$
are sparse,
which can be largely exploited to develop
better methods for variable selection.
On the other hand, ``there is no free lunch,'' and sparsifying
also causes serious issues:
\begin{itemize}
\item The post-filtering model (\ref{modeladd}) is not a regular
linear regression model.
\item If we apply a local method (e.g., UPS, forward/backward
regression) to model~(\ref{modeladd}), we face so-called challenge of
\textit{information leakage}.
\end{itemize}
In Section~\ref{subsec:patch}, we carefully
explain the issue of information leakage, and discuss how to deal with it.

We remark that while sparsifying can be very helpful, it does not mean
that it is trivial to derive optimal procedures from model (\ref{modeladd}).
For example, if we apply the $L^0$-penalization method naively to model
\eqref{modeladd},
we then ignore the correlations among the noise, which cannot be optimal.
If we apply the $L^0$-penalization method with the correlation
structures incorporated, we are
essentially applying it to the original regression model (\ref{model}).

\subsection{Covariate assisted screening and estimation (CASE)}
To exploit the sparsity in $DG$ and $DGD'$, and to deal with the two
aforementioned issues that
sparsifying causes, we propose a new variable selection method which we call
\textit{covariate assisted screening and estimation} (CASE).
The main methodological innovation of CASE is
to use linear filtering to create graph sparsity
and then to exploit the rich information hidden in the ``local''
graphical structures among
the design variables, which the lasso and many other procedures do not utilize.

At the heart of CASE is \textit{covariate assisted} multivariate screening.
Screening is a well-known method of dimension reduction in big data.
However, most literature to date has been focused on \textit{univariate
screening} or \textit{marginal screening} [\citet{FanLv,Genovese}].
Extending marginal screening to (brute-force) $m$-variate screening, $m
> 1$, means that we screen
all ${p\choose m}$ size-$m$ sub-models, and has two major concerns:
\begin{itemize}
\item\textit{Computational infeasibility}. A brute-force $m$-variate
screening has a computation
complexity of $O(p^m)$, which is usually not affordable.
\item\textit{Screening inefficiency}. The goal of screening is to
remove as many noise entries as we can while
retaining most of the signals. When we screen too many submodels than
necessary, we have
to set the bar higher than necessary to exclude most of the noise
entries. As a result, we need
signals stronger than necessary in order for them to survive the screening.
\end{itemize}
To overcome these challenges, CASE uses a new screening strategy called
covariance-assisted screening, which
excludes most size-$m$ submodels from
screening but still manages to retain almost all signals.
In detail, we first use the Gram matrix $G$ to construct a sparse graph
called \textit{graph of strong dependence} (GOSD).
We then include a size-$m$ submodel in our screening list if and only
if the $m$ nodes form a connected subgraph of GOSD. This way, we
exclude many submodels from the screening by only using information in
$G$, not that in the response vector $Y$!

The blessing is, when GOSD is sufficiently sparse, it has no more than
$L_p p$ connected size-$m$ sub-graphs, where $L_p$ is a generic
multi-$\log(p)$ term. Therefore, covariance-assisted screening only
visits $L_p p$ submodels,
in contrast to ${p\choose m}$ submodels the brute-forth screening
visits. As a result, covariance-assisted screening is not only computationally
feasible, but is also efficient. Now, it would not be a surprise that
CASE is a ``fundamentally correct'' procedure in the rare/weak regime,
at least when the GOSD is sufficiently sparse, as in settings
considered in this paper; see more discussion below.

\subsection{Objective of the theoretic study}
We now discuss the theoretic component of the paper. The objective of
our theoretic study is three-fold:
\begin{itemize}
\item to develop a theoretic framework that is appropriate for the
regime where signals are rare/weak, and $G$ is nonsparse but is sparsifiable;
\item to appreciate the ``pros'' and ``cons'' of sparsifying, and to
investigate how to fix the ``cons'';
\item to show that CASE is asymptotic minimax and yields an optimal
partition of the so-called \textit{phase diagram}.
\end{itemize}
The phase diagram is a relatively new criterion for assessing
the optimality of procedures. Call the two-dimensional
space calibrated by the \textit{signal rarity} and \textit{signal strength}
the phase space. The phase diagram is the partition of the phase
space into different regions where in each of them inference
is distinctly different. The notion of phase diagram is especially appropriate
when signals are rare and weak.

The theoretic study is challenging for many reasons:
\begin{itemize}
\item We focus on a very challenging regime, where signals are rare and
weak, and the design matrix is nonsparse or even ill-posed. Such a
regime is important from a practical perspective, but has not been
carefully explored in the literature.
\item The goal of the paper is to develop procedures in the rare/weak
regime that are asymptotic minimax in terms of Hamming distance, to
achieve which we need to find a lower bound and an upper bound
that are both tight. Compared to most works on variable selection where
the goal is to find procedures that yield exact recovery for
sufficiently strong signals, our goal is comparably
more ambitious, and the study it entails is more delicate.
\item To derive the phase diagrams in Sections~\ref{subsec:Toeplitz}--\ref{subsec:chang}, we need explicit forms of
the convergence
rate of minimax Hamming selection errors. This usually
needs very delicate analysis. The study associated with the
change-point model is especially
challenging and long.
\end{itemize}

\subsection{Content and notation}
The paper is organized as follows. Section~\ref{sec:Main} contains the
main results of this paper: we formally introduce CASE and establishe
its asymptotic optimality. Section~\ref{sec:Simul} contains simulation
studies, and Section~\ref{sec:Discu} contains conclusions and discussions.

Throughout this paper, $D=D_{h,\eta}$, $d=D(X'Y)$, $B=DG$, $H=DGD'$ and
${\cal G}^*$ denotes the GOSD (in contrast, $d_p$ denotes the degree of
GOLF, and $H_p$ denotes the Hamming distance). Also, $\mathbb{R}$ and
$\mathbb{C}$ denote the sets of real numbers and complex numbers,
respectively, and $\mathbb{R}^p$ denotes the $p$-dimensional real
Euclidean space. Given $0\leq q\leq\infty$, for any vector $x$,
$\| x \|_q$ denotes the $L^q$-norm of $x$; for any matrix $M$,
$\| M \|_q$ denotes the matrix $L^q$-norm of $M$. When $q = 2$, $\|M\|
_q$ coincides with the matrix spectral norm; we shall omit the
subscript $q$ in this case.
When $M$ is symmetric, $\lambda_{\max}(M)$ and $\lambda_{\min}(M)$
denote the maximum and minimum eigenvalues of $M$, respectively. For two
matrices $M_1$ and $M_2$, $M_1\succeq M_2$ means that $M_1-M_2$ is
positive semi-definite.

\section{Main results} \label{sec:Main}
This section is arranged as follows.
Sections~\ref{subsec:RW}--\ref{subsec:compu} focus on the model, ideas
and the method.
In Section~\ref{subsec:RW}, we introduce the rare and weak signal
model. In Section~\ref{subsec:linea}, we formally introduce the notion
of \textit{sparsifiability}. The starting point of CASE is the use of a
linear filter.
In Section~\ref{subsec:LFbenefit}, we explain how linear filtering
helps in variable selection by
inducing a sparse graph and an interesting interaction between the
graphical sparsity and the signal sparsity.
In Section~\ref{subsec:patch}, we explain that linear filtering also
causes a so-called
problem of \textit{information leakage}, and discuss how to overcome
such a problem by the technique
of \textit{patching}. After all these ideas are discussed, we formally
introduce the CASE in Section~\ref{subsec:ags}.
In Section~\ref{subsec:compu}, we discuss the computational complexity
and show that CASE is computationally feasible
in a broad context.

Sections~\ref{subsec:minimax}--\ref{subsec:main} focus on the
asymptotic optimality of CASE.
In Section~\ref{subsec:minimax}, we
introduce the asymptotic minimax framework where we use Hamming
distance as the loss function.
In Section~\ref{subsec:LB}, we study the lower bound for the
minimax Hamming risk, and in Section~\ref{subsec:main}, we show that
CASE achieves the minimax Hamming risk in a broad context.

In Sections~\ref{subsec:Toeplitz}--\ref{subsec:chang}, we apply our
results to long-memory time series and the change-point model. For both
of them, we first derive explicit formulas for the convergent rates,
and then use the formulas to derive the phase diagrams.

Proofs of results in this section can be found in the supplemental article
[\citet{CASEsupp}], which contains Sections A--C. 

\subsection{Rare and weak signal model} \label{subsec:RW}
Our primary interest is in the situations where the signals are rare
and weak, and where we have \textit{no} information on the underlying
structure of the signals. In such situations, it makes
sense to use the following \textit{rare and weak} signal model; see
\citet{Candes,DonohoJin2008,JZ11}. Fix $\eps\in(0,1)$ and $\tau>
0$. Let $b = (b_1, \ldots, b_p)'$ be the $p \times1$ vector satisfying
%
\begin{equation}
\label{Defineb} b_i \stackrel{\mathrm{i.i.d.}} {\sim} \operatorname{Bernoulli}(
\eps),
\end{equation}
and let $\Theta_p(\tau)$ be the set of vectors
%
\begin{equation}
\label{DefineTheta} \Theta_p(\tau) = \bigl\{ \mu\in\mathbb{R}^p
\dvtx |\mu_i| \geq\tau, 1 \leq i \leq p\bigr\}.
\end{equation}
We model $\beta$ by
%
\begin{equation}
\label{sparse signal model on beta} \beta= b \circ\mu,
\end{equation}
where $\mu\in\Theta_p(\tau)$ and $\circ$ is the Hadamard product
(also called the coordinate-wise product). In Section~\ref{subsec:minimax}, we further restrict $\mu$ to a subset of $\Theta
_p(\tau)$.

In this model, $\beta_i$ is either $0$ or a signal with a strength
$\geq\tau$. Since we have no information on where the signals are, we assume
that they appear at locations that are randomly generated.
We are primarily interested in the challenging case where $\eps$ is small
and $\tau$ is relatively small, so the signals are both rare and weak.

\begin{definition}
We call model \eqref{Defineb}--\eqref{sparse signal model on beta} the
rare and weak signal model RW$(\eps, \tau, \mu)$.
\end{definition}

We remark that the theory developed in this paper is not tied to the
rare and weak signal model, and applies to more general cases. For
example, the main results can be extended
to the case where we have some additional information about the
underlying structure of the signals (e.g., Ising's model [\citet{Ising}]).

\subsection{Sparsifiability, linear filtering and GOSD} \label{subsec:linea}
As mentioned before, we are primarily interested in the case where the
Gram matrix $G$ can be sparsified
by a finite-order linear filtering.

Fix an integer $h \geq1$ and an $(h+1)$-dimensional vector $\eta= (1,
\eta_1, \ldots, \eta_h)'$. Let $D = D_{h, \eta}$ be the $p \times p$
matrix satisfying
%
\begin{eqnarray}
\label{DefineD} D_{h, \eta}(i,j) = 1\{ i = j\} + \eta_1 1\{ i =
j -1\} + \cdots+ \eta _h 1\{i = j - h\},
\nonumber
\\[-8pt]
\\[-8pt]
 \eqntext{1 \leq i, j \leq p.}
\end{eqnarray}
The matrix $D_{h, \eta}$ can be viewed as a linear operator that maps
any $p \times1$ vector
$y$ to $D_{h, \eta} y$. For this reason, $D_{h, \eta}$ is also called
an order $h$ linear filter [\citet{FanTimeS}].

For $\alpha> 0$ and $A_0 > 0$, we introduce the following class of matrices:
%
\begin{eqnarray}
\label{matrixclass}&& \mathcal{M}_p(\alpha,A_0)=\bigl\{ \Omega
\in\mathbb{R}^{p\times p}\dvtx \Omega (i,i)\leq1, \bigl|\Omega(i,j)\bigr|\leq
A_0\bigl(1 + |i-j|\bigr)^{-\alpha},
\nonumber
\\[-8pt]
\\[-8pt]
\nonumber
&& \hspace*{258pt}1 \leq i, j \leq p \bigr\}.
\end{eqnarray}
Matrices in $\mathcal{M}_p(\alpha,A_0)$ are not necessarily symmetric.
%
\begin{definition}\label{def:sparse}
Fix an order $h$ linear filter $D = D_{h, \eta}$. We say that $G$ is
sparsifiable by $D_{h, \eta}$ if for sufficiently large $p$, $DG\in
\mathcal{M}_p(\alpha, A_0)$ for some constants $\alpha>1$ and $A_0 > 0$.
\end{definition}
In the long-memory time series model, $G$ can be sparsified by an order
$1$ linear filter. In the change-point model, $G$ can be sparsified by
an order $2$ linear filter.

The main benefit of linear filtering is that it induces sparsity in the
graph of strong dependence (GOSD) to be introduced below. Recall that
the sufficient statistics $\ty= X'Y \sim N(G \beta, G)$.
Applying a linear filter $D = D_{h, \eta}$ to $\ty$ gives
%
\begin{equation}
\label{dmodel} d \sim N(B \beta, H),
\end{equation}
where $d = D (X'Y)$, $B = DG$ and $H = D G D'$.
Note that no information is lost when we reduce from the model $\ty
\sim
N(G\beta, G)$ to model \eqref{dmodel}, as $D$ is a nonsingular matrix.

At the same time, if $G$ is sparsifiable by $D = D_{h, \eta}$, then
both the matrices $B$ and $H$ are sparse, in the sense that each row of
either matrix has relatively few large coordinates. In other words, for
a properly small threshold $\delta> 0$ to be determined, let
$B^*$ and $H^*$ be the regularized matrices of $B$ and $H$, respectively,
\begin{eqnarray*}
B^*(i,j) &= &B(i,j) 1\bigl\{\bigl |B(i, j)\bigr| \geq\delta\bigr\},\\
 H^*(i,j) &=& H(i,j) 1\bigl
\{ \bigl|H(i,j)\bigr| \geq\delta\bigr\},\qquad 1 \leq i, j \leq p.
\end{eqnarray*}
It is seen that
%
\begin{equation}
\label{dmodel1} d \approx N\bigl(B^* \beta, H^*\bigr),
\end{equation}
where each row of $B^*$ or $H^*$ has relatively few nonzeros.
Compared to \eqref{dmodel}, \eqref{dmodel1} is much easier to track
analytically, but
it contains
almost all the information about~$\beta$.

The above observation naturally motivates the following graph, which we
call the \textit{graph of strong dependence} (GOSD).
%
\begin{definition}\label{DefineGOSD}
For a given parameter $\delta$, the GOSD is the
graph $\cG^* = (V, E)$ with nodes $V = \{1, 2, \ldots, p\}$, and there
is an edge between $i$ and $j$ when
any of the three numbers $H^*(i,j)$, $B^*(i,j)$ and $B^*(j,i)$ is nonzero.
\end{definition}

%
\begin{definition}
A graph ${\cal G}= (V, E)$ is called $K$-sparse if the degree of each node
$\leq K$.
\end{definition}
The definition of GOSD depends on a tuning parameter $\delta$, the
choice of which is not critical, and it is generally sufficient if we
choose $\delta= \delta_p = O(1/\log(p))$; see Section B.1 in \citet{CASEsupp}%
for details. With such a choice of $\delta$, it can be shown that in a
general context, GOSD is $K$-sparse, where $K = K_{\delta}$ does not
exceed a multi-$\log(p)$ term as $p \goto\infty$; see Lemma B.1 in \citet{CASEsupp}. %

\subsection{Interplay between the graph sparsity and signal sparsity}
\label{subsec:LFbenefit}
With these being said, it remains unclear how the sparsity of ${\cal G}^*$
helps in variable selection. In fact, even when ${\cal G}^*$ is $2$-sparse,
it is possible that a node $k$ is connected---through possible long
paths---to many
other nodes; it is unclear how to remove the effect of these nodes
when we try to estimate $\beta_k$.

Somewhat surprisingly, the answer lies in an interesting interplay
between the signal sparsity
and graph sparsity.
To see this point, let $S = S(\beta)$ be the support of $\beta$, and
let ${\cal G}_S^*$ be the subgraph of ${\cal G}^*$ formed by the nodes
in $S$ only.
Given the sparsity of $\cG^*$, if the signal vector $\beta$ is also
sparse, then it is likely
that the sizes of all components of ${\cal G}_S^*$ (a component of a graph
is a maximal connected subgraph) are uniformly small. This is justified
in the following lemma which is proved in \citet{JZ11}.
%
\begin{lemma} \label{lemma:split}
Suppose ${\cal G}^*$ is $K$-sparse, and the support $S = S(\beta)$ is a
realization from $\beta_j\stackrel{\mathrm{i.i.d.}}{\sim}(1-\varepsilon) \nu_0 +
\varepsilon\pi$, where $\nu_0$ is the point mass at $0$ and $\pi$ is any
distribution with support $\subseteq\mathbb{R} \setminus\{0\}$. With
a probability (from randomness of $S$) at least $1 - p (e \eps K)^{m
+1}$, ${\cal G}_S^*$ decomposes into many components with size no larger
than $m$.
\end{lemma}
In this paper, we are primarily interested in cases where for large
$p$, $\eps\leq p^{-\vartheta}$ for some
parameter $\vartheta\in(0, 1)$ and $K$ is bounded by a multi-$\log
(p)$ term. In such cases, the decomposability of ${\cal G}_S^*$ holds
for a
finite $m$, with overwhelming probability.

Lemma~\ref{lemma:split} delineates an interesting picture: The set of
signals decomposes into many small-size isolated ``signal islands'' (if
only we know where), each of them is a component of ${\cal G}_S^*$ and
different ones are disconnected in the GOSD. As a result, the original
$p$-dimensional problem can be viewed as the aggregation of many
separated small-size subproblems that can be solved parallelly. This is
the key insight of this paper.

Note that the decomposability of ${\cal G}_S^*$ attributes to the interplay
between the signal sparsity and the graph sparsity, where the latter
attributes to the use of linear filtering. The decomposability is not
tied to the specific model of $\beta$ in Lemma~\ref{lemma:split}, and
holds for much broader situations (e.g., when $b$ is generated by a
sparse Ising model [\citet{Ising}]).

\subsection{Information leakage and patching} \label{subsec:patch}
While it largely facilitates the decomposability of the model, we must
note that the linear filtering also induces a so-called problem of
\textit{information leakage}. In this section, we discuss how linear
filtering causes such a problem and how to overcome it by the so-called
technique of \textit{patching}.

The following notation is frequently used in this paper.
%
\begin{definition}
For $\call\subset\{1, 2,\ldots, p\}$, $\callJ\subset\{1,\ldots,
N\}
$ and a $p \times N$ matrix~$X$, $X^{\call}$ denotes the $|\call
|\times
N$ sub-matrix formed by restricting the rows of $X$ to~$\call$, and
$X^{\call, \callJ}$ denotes the $|\call|\times|\callJ|$ sub-matrix
formed by restricting the rows of $X$ to $\call$ and columns to
$\callJ$.
\end{definition}

Note that when $N = 1$, $X$ is a $p \times1$ vector, and $X^{\call}$
is an $|\call| \times1$ vector.

To explain information leakage, we first consider an idealized case
where each row of $G$ has $\leq K$ nonzeros. In this case, there is no
need for linear filtering, so $B = H = G$ and $d = \ty$. Recall that
${\cal G}^*_S$ consists of many signal islands, and let $\call$ be one of
them. It is seen that
%
\begin{equation}
\label{addsubmodel} d^{\call} \approx N \bigl(G^{\call,\call}
\beta^{\call}, G^{\call,\call}\bigr),
\end{equation}
and how well we can estimate $\beta^{\call}$ is captured by the Fisher
information matrix $G^{\call, \call}$ [\citet{TPE}].

Come back to the case where $G$ is nonsparse. Interestingly, despite
the strong correlations, $G^{\call, \call}$ continues to be the Fisher
information for estimating $\beta^{\call}$. However, when $G$ is
nonsparse, we must use a linear filtering $D = D_{h, \eta}$ as
suggested, and we have
%
\begin{equation}
\label{dsubmodel} d^{\call} \approx N \bigl(B^{\call,\call}
\beta^{\call}, H^{\call,\call}\bigr).
\end{equation}
Moreover, letting $\callJ= \{1 \leq j \leq p\dvtx \mbox{$D(i,j) \neq0$
for some $i \in\call$} \}$,
it follows that
\[
B^{\call,\call}\beta^{\call} = D^{\call, \callJ} G^{\callJ,
\call} \beta
^{\call}.
\]
By the definition of $D$, $|\callJ| > |\call|$ and the dimension of the
following null space $\geq1$,
%
\begin{equation}
\label{esubmodel} \operatorname{Null}(\call, \callJ) = \bigl\{ \xi\in\mathbb{R}^{|\callJ|}
\dvtx D^{\call,
\callJ} \xi= 0 \bigr\}.
\end{equation}
Compare \eqref{dsubmodel} with \eqref{addsubmodel}, and imagine
the oracle situation where we are told the mean vector of $d^{\call}$
in both. The difference is that we can fully recover
$\beta^{\call}$ using~(\ref{addsubmodel}), but are not able to do so
with only (\ref{dsubmodel}).
In other words, the information containing $\beta^{\call}$ is partially
lost in (\ref{dsubmodel}): if we estimate $\beta^{\call}$ with (\ref
{dsubmodel}) alone, we will never
achieve the desired accuracy.

The argument is validated in Lemma~\ref{lem:FIM} below, where
the Fisher information associated with (\ref{dsubmodel}) is strictly
``smaller'' than $G^{\call, \call}$; the difference between two matrices
can be derived by taking $\call^+ = \call$ and ${\cal J}^+ = \callJ$ in
(\ref{FIM2}).
We call this phenomenon ``information leakage.''

To mitigate this, we expand the information content by including data
in the neighborhood of $\call$. This process is called ``patching.''
Let $\call^+$ be an extension of $\call$ by adding a few neighboring
nodes, and define similarly $\callJ^+=\{1\leq j\leq p\dvtx D(i,j)\neq0
\mbox{ for some } i\in\call^+\}$ and $\operatorname{Null}(\call^+, \callJ^+)$.
Assuming that there is no edge between any node in $\call^+$ and any
node in ${\cal G}_S^* \setminus\call$,
%
\begin{equation}
\label{dplusmodel} d^{\call^+} \approx N\bigl( B^{\call^+, \call}
\beta^{\call}, H^{\call^+,
\call^+}\bigr).
\end{equation}
The Fisher information matrix for $\beta^{\call}$ under model \eqref
{dplusmodel} is larger than that of~\eqref{dsubmodel}, which is
captured in the following lemma.
%
\begin{lemma} \label{lem:FIM}
The Fisher information matrix associated with model \eqref{dplusmodel} is
%
\begin{equation}
\label{FIM2} G^{\call, \call} - \bigl[ U\bigl(U'
\bigl(G^{\callJ^+, \callJ^+}\bigr)^{-1} U \bigr)^{-1}
U' \bigr]^{\call,\call},
\end{equation}
where\vspace*{1pt} $U$ is any $|\callJ^+| \times(|\callJ^+| - |\call^+|)$ matrix
whose columns form
an orthonormal basis of $\operatorname{Null}(\call^+, \callJ^+)$.
\end{lemma}

When the size of $\call^+$ becomes appropriately large,
the second matrix in (\ref{FIM2}) is small element-wise (and so is
negligible) under mild conditions [see details in Lemma~A.3 in
\citet{CASEsupp}]. This matrix is usually nonnegligible if we set $\call^+=\call$
and $\callJ^+ = \callJ$ (i.e., without patching).

\begin{example}\label{ex1}
We illustrate the above phenomenon with an example
where $p = 5000$, $G$ is the matrix satisfying
$G(i,j) = [1 +5 |i - j|]^{-0.95}$ for all $1 \leq i, j \leq p$ and
$D = D_{h, \eta}$ with $h = 1$ and $\eta= (1,-1)'$.
If $\call= \{2000\}$, then $G^{\call, \call} = 1$, but the Fisher
information associated with model (\ref{dsubmodel}) is $0.5$. The gap
can be substantially narrowed if we patch with $\call^+ = \{1990, 1991,\ldots,\break 2010\}$, in which case the Fisher information in
model (\ref{dplusmodel}) is $0.904$.
\end{example}

Although one of the major effects of information leakage is a reduction
in the signal-to-noise ratio, this phenomenon is very different from
the well-known ``signal cancellation'' or ``partial faithfulness'' in
variable selection. ``Signal cancellation'' is caused by correlations
between signal covariates, and CASE overcomes this problem by using
multivariate screening. However, ``information leakage'' is caused by
the use of a linear filtering. From Lemma~\ref{lem:FIM}, we can see that the
information leakage appears no matter for what signal vector $\beta$.
CASE overcomes this problem by the patching technique.

\subsection{Covariate assisted screening and estimation (CASE)}
\label{subsec:ags}
In summary, we start from the post-filtering regression model
\[
d = D\ty\qquad \mbox{where $\ty= X'Y$ and $D = D_{h, \eta}$ is a
linear filter}.
\]
We have observed the following:
\begin{itemize}
\item\textit{Signal decomposability}. Linear filtering induces
sparsity in GOSD, a graph
constructed from the Gram matrix $G$. In this graph,
the set of all true signals decomposes into many small-size
signal islands, each signal island is a component of GOSD.
\item\textit{Information patching}. Linear filtering also causes
information leakage, which can be overcome
by delicate patching technique.
\end{itemize}
Naturally, these motivate a two-stage screen and clean approach for
variable selection, which we call \textit{covariate assisted screening
and estimation} (CASE). CASE contains a \textit{patching and
screening} (PS) step and a \textit{patching and estimation} (PE) step.
\begin{itemize}
\item\textit{$\mathit{PS}$-step}. We use sequential $\chi^2$-tests to identify
candidates
for each signal island. Each $\chi^2$-test is guided by ${\cal G}^*$, and
aided by a carefully designed patching step. This achieves multivariate
screening without visiting all submodels.
\item\textit{$\mathit{PE}$-step}. We re-investigate each candidate with
penalized MLE and
certain patching technique, in hopes of removing false positives.
\end{itemize}


For the purpose of patching,
the $\mathit{PS}$-step and the $\mathit{PE}$-step use tuning integers $\ell^{ps}$ and
$\ell^{pe}$, respectively. The following notation is frequently used in
this paper.
%
\begin{definition}\label{DefinePatchingSet}
For any index $1 \leq i \leq p$, $\{i\}^{ps} = \{1 \leq j \leq p\dvtx |j -
i| \leq\ell^{ps} \}$. For any
subset $\call$ of $\{1, 2,\ldots, p\}$, $\call^{ps} = \bigcup_{i \in
\call
} \{i\}^{ps}$. Similar notation applies to $\{i\}^{pe}$ and $\call^{pe}$.
\end{definition}

We now discuss two steps in detail. Consider the $\mathit{PS}$-step first.
Fix $m > 1$. Suppose that ${\cal G}^*$ has a total of $T$ connected
subgraphs with size $\leq m$, which we denote by $\{{\cal G}_t\}_{t =
1}^{T}$, arranged in the ascending order of the sizes, with ties
breaking lexicographically.

\renewcommand{\theexample}{\arabic{example}\normalfont{(a)}}
\begin{example}\label{ex2a}
We illustrate this with a toy example, where $p
= 10$ and the GOSD is displayed in Figure~\ref{Fig1}(a). For $m = 3$,
GOSD has $T = 30$ connected subgraphs, which we arrange as follows.
Note that $\{{\cal G}_t\}_{t = 1}^{10}$ are singletons, $\{ {\cal G}_t
\}_{t =
11}^{20}$ are connected pairs and $\{{\cal G}_t\}_{t = 21}^{30}$ are
connected triplets
\begin{eqnarray*}
&& \{1\}, \{2\}, \{3\}, \{4\}, \{5\}, \{6\}, \{7\}, \{ 8\}, \{9\}, \{10\},
\\
&& \{1,2\}, \{1,7\}, \{2,4\}, \{3,4\}, \{4,5\}, \{ 5,6\} , \{7,8\}, \{8,9\},
\{8,10\}, \{9,10\},
\\
&& \{1,2,4\}, \{1,2,7\}, \{1,7,8\}, \{2,3,4\}, \{ 2,4,5\}, \{3,4,5\}, \{4,5,6\},
\{7,8,9\},
\\
&& \{7,8,10\}, \{8,9,10\}.
\end{eqnarray*}

Here we examine sequentially only the 30 submodels above to decide
whether any variables have additional utilities given the variables
recruited before, via $\chi^2$-tests. The first 10 screening problems
are just the univariate screening. After that, starting from bivariate
screening, we examine the variables given those selected so far.
Suppose that we are examining the submodel $\{1, 2\}$. The testing
problem depends on how the variables $\{1, 2\}$ are selected in the
previous steps. For example, if the variables $\{1,2, 4, 6\}$ have
already been selected in the univariate screening, there is no new
recruitment, and we move on to examine the submodel $\{1, 7\}$. If the
variables $\{1, 4, 6\}$ have been recruited so far, we need to test if
variable $\{2\}$ has additional contributions given variable $\{1\}$.
If the variables $\{4, 6\}$ have been recruited in the previous steps,
we will examine whether variables $\{1, 2\}$ together have any
significant contributions. Therefore, we have never run regression for
more than two variables. Similarly, for trivariate screening, we will
never run regression for more than 3 variables. Clearly, multivariate
screening improves the marginal screening in that it gives signal
variables chances to be recruited if they are wrongly excluded by the
marginal method.
\end{example}

We now formally describe the procedure. The $\mathit{PS}$-step contains $T$
sub-stages, where we screen
${\cal G}_t$ sequentially, $t = 1, 2,\ldots, T$. Let
${\cal U}^{(t)}$ be the set of retained indices at the end of stage
$t$, with ${\cal U}^{(0)} = \varnothing$ as the convention. For $1 \leq t
\leq T$, the $t$th sub-stage contains two sub-steps:
\begin{itemize}
\item(\textit{Initial step}). Let $\hat{N} = {\cal U}^{(t-1)} \cap
{\cal G}_{t}$
represent the set of nodes in ${\cal G}_t$ that have already been accepted
by the end of the $(t-1)$th sub-stage, and let $\hat{F} = {\cal G}_t
\setminus\hat{N}$ be the set of other nodes in ${\cal G}_t$.
\item(\textit{Updating step}). Write for short $\call= {\cal G}_t$. Fixing
a tuning parameter $\ell^{ps}$ for patching, introduce
%
\begin{eqnarray}
\label{def of W and Q(ps)} W& =&\bigl(B^{\call^{ps},\call}\bigr)'
\bigl(H^{\call^{ps},\call^{ps}}\bigr)^{-1} d^{\call^{ps}},
\nonumber
\\[-8pt]
\\[-8pt]
\nonumber
Q &=&
\bigl(B^{\call^{ps}, \call}\bigr)'\bigl(H^{\call^{ps},\call^{ps}}
\bigr)^{-1} \bigl(B^{\call
^{ps},\call}\bigr),
\end{eqnarray}

\begin{figure}

\includegraphics{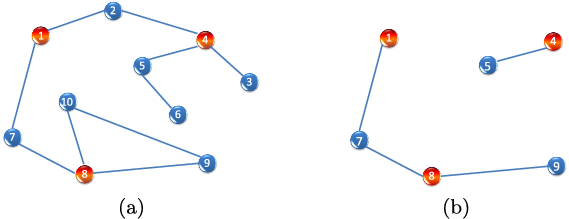}

\caption{Illustration of graph of strong dependence (GOSD). Red: signal
nodes. Blue: noise nodes.
\textup{(a)} GOSD with 10 nodes.
\textup{(b)} Nodes of GOSD that survived the $\mathit{PS}$-step.
} \label{Fig1}
\end{figure}
%
%

\noindent where $W$ is a random vector and $Q$ can be thought of as the
covariance matrix of $W$.
Define $W_{\hat{N}}$, a subvector of $W$, and $Q_{\hat{N},\hat{N}}$, a
submatrix of~$Q$, as follows:
%
\begin{eqnarray}
\label{WNadd} W_{\hat{N}}&= &\bigl(B^{\call^{ps},\hat{N}}\bigr)'
\bigl(H^{\call^{ps},\call^{ps}}\bigr)^{-1} d^{\call^{ps}},
\nonumber
\\[-8pt]
\\[-8pt]
\nonumber
 Q_{\hat{N},\hat{N}}&=&
\bigl(B^{\call^{ps}, \hat{N}}\bigr)'\bigl(H^{\call
^{ps},\call
^{ps}}
\bigr)^{-1} \bigl(B^{\call^{ps},\hat{N}}\bigr).
\end{eqnarray}
Introduce the test statistic
%
\begin{equation}
\label{def of T statistics} T(d, \hat{F}, \hat{N}) = W' Q^{-1}W -
W_{\hat{N}}'(Q_{\hat{N},\hat
{N}})^{-1}
W_{\hat{N}}.
\end{equation}
For a threshold $t = t(\hat{F}, \hat{N})$ to be determined, we update
the set of retained nodes by ${\cal U}^{(t)} = {\cal U}^{(t-1)} \cup
\hat{F}$ if $T(d, \hat{F}, \hat{N}) > t$, and let ${\cal U}^{(t)} =
{\cal U}^{(t-1)}$ otherwise. In other words, we accept nodes in $\hat
{F}$ only when they have additional utilities.
\end{itemize}
The $\mathit{PS}$-step terminates at $t = T$. We then write ${\cal U}_p^* =
{\cal U}^{(T)}$ so that
\[
{\cal U}_p^* = \mbox{the set of all retained indices at the end of
the $\mathit{PS}$-step}.
\]

In the $\mathit{PS}$-step, as we screen, we accept nodes sequentially. Once a
node is accepted in the $\mathit{PS}$-step,
it stays there until the end of the $\mathit{PS}$-step; of course, this node
could be killed in the $\mathit{PE}$-step. In spirit, this is similar to the well-known
forward regression method, but the implementation of two methods are
significantly different.

The $\mathit{PS}$-step uses a collection of tuning thresholds
\[
\tq= \bigl\{ t(\hat{F}, \hat{N})\dvtx \mbox{$(\hat{F}, \hat{N})$ are defined
above} \bigr\}.
\]
A convenient choice for these thresholds is to let
$t(\hat{F}, \hat{N}) = 2 \tilde{q} \log(p) |\hat{F}|$ for a properly
small fixed constant $\tilde{q} > 0$. See Section~\ref{subsec:main}
(and also Sections~\ref{subsec:Toeplitz}--\ref{subsec:chang}) for more
discussion on the choices of $t(\hat{F}, \hat{N})$.


In the $\mathit{PS}$-step, we use $\chi^2$-test for screening. This is the best
choice when the coordinates of $z$ are Gaussian and have the same
variance. When the Gaussian assumption on $z$ is questionable, we must
note that the $\chi^2$-test depends on the Gaussianity of $a'z$ for all
$p$-different $a$, not on that of $z$;
$a'z$ could be approximately Gaussian by central limit theorem.
Therefore, the performance of $\chi^2$-test is relatively robust to
non-Gaussianity. If circumstances arise that the $\chi^2$-test is not
appropriate (e.g., misspecification of the model, low quantity of the
data), we may need an alternative, say, some nonparametric tests. In
this case, if the efficiency of the test is nearly optimal, then the
screening in the $\mathit{PS}$-step would continue to be successful.

How does the $\mathit{PS}$-step help in variable selection? In Section A in \citet{CASEsupp}, %
we show that in a broad context, provided that the tuning parameters
$t(\hat{F}, \hat{N})$ are properly set,
the $\mathit{PS}$-step has two noteworthy properties: the \textit{sure
screening} (SS) property and the \textit{separable after
screening} (SAS) property. The SS property says that $\cU_p^*$ contains
all but a negligible
fraction of the true signals. The SAS property says that if we view
${\cal U}_p^*$ as a subgraph of
${\cal G}^*$ (more precisely, as a subgraph of ${\cal G}^+$, an
expanded graph of
${\cal G}^*$ to be introduce below),
then this subgraph decomposes into many disconnected components, each
having a moderate size.

Together, the SS property and the SAS property enable us to reduce the
original large-scale problem to many parallel small-size
regression problems, and pave the way for the $\mathit{PE}$-step.
See Section A in \citet{CASEsupp}
for details.

\renewcommand{\theexample}{\arabic{example}\normalfont{(b)}}
\setcounter{example}{1}
\begin{example}
 We illustrate the above points with the toy
example in Example~\ref{ex2a}. Suppose after the $\mathit{PS}$-step, the set of
retained indices ${\cal U}_p^*$ is $\{1, 4, 5, 7, 8, 9\}$; see
Figure~\ref{Fig1}(b). In this example, we have a total of three signal
nodes, $\{1\}$, $\{4\}$ and $\{8\}$, which
are all retained in ${\cal U}_p^*$ and so the $\mathit{PS}$-step yields sure
screening. On the other hand, ${\cal U}_p^*$ contains
a few nodes of false positives, which will be further cleaned in the
$\mathit{PE}$-step. At the same time,
viewing it as a subgraph of ${\cal G}^*$, ${\cal U}_p^*$ decomposes
into two
disconnected components, $\{1, 7, 8, 9\}$ and $\{4, 5\}$;
compare Figure~\ref{Fig1}(a). The SS property and the SAS property
enable us to reduce the original problem of $10$ nodes to two parallel
regression problems,
one with $4$ nodes, and the other with $2$ nodes.
\end{example}

We now discuss the $\mathit{PE}$-step. Recall that $\ell^{pe}$ is the tuning
parameter for the
patching of the $\mathit{PE}$-step, and let $\{i\}^{pe}$ be as in Definition~\ref
{DefinePatchingSet}. The following graph
can be viewed as an expanded graph of ${\cal G}^*$.

%
\begin{definition}\label{DefineGplus}
Let ${\cal G}^+ = (V, E)$ be the graph where $V = \{1, 2,\ldots, p\}$ and
there is an edge between nodes $i$ and $j$ when there exist nodes $k
\in\{i\}^{pe}$ and $k' \in\{j\}^{pe}$ such that there is an edge
between $k$ and $k'$ in ${\cal G}^*$.
\end{definition}

Recall that ${\cal U}_p^*$ is the set of retained indices at the end of
the $\mathit{PS}$-step.
%
\begin{definition}\label{def:subgraphnotation}
Fix a graph ${\cal G}$ and its subgraph $\call$. We say $\call\unlhd
{\cal G}$
if $\call$ is a connected subgraph of ${\cal G}$, and $\call\lhd
{\cal G}$ if
$\call$ is a component (maximal connected subgraph) of ${\cal G}$.
\end{definition}
Fix $1 \leq j \leq p$. When $j \notin{\cal U}_p^*$, CASE estimates
$\beta_j$ as $0$. When $j \in{\cal U}_p^*$, viewing ${\cal U}_p^*$ as
a subgraph of
${\cal G}^+$, there is a unique subgraph $\call$ such that
$j \in\call\lhd{\cal U}_p^*$.
Fix two tuning parameters $u^{pe}$ and $v^{pe}$. We estimate $\beta
^{\call}$ by minimizing
%
\begin{eqnarray}
\label{obj of minimization}&& \min_\theta \biggl\{ \frac{1}{2}
\bigl(d^{\call^{pe}}-B^{\call
^{pe},\call
}\theta\bigr)'
\bigl(H^{\call^{pe},\call^{pe}}\bigr)^{-1}\bigl(d^{\call^{pe}}-B^{\call
^{pe},\call}
\theta\bigr)
\nonumber
\\[-8pt]
\\[-8pt]
\nonumber
&&\hspace*{180pt}{}+\frac{(u^{pe})^2}{2}\Vert\theta\Vert_0 \biggr\},
\end{eqnarray}
subject to that $\theta$ is an $|\call| \times1$ vector each of which
nonzero coordinate $\geq v^{pe}$, where $\|\theta\|_0$ denotes the
$L^0$-norm of $\theta$.
Putting these together gives the final estimator of CASE $\hat{\beta
}^{\mathrm{case}} = \hat{\beta}^{\mathrm{case}}(Y; \delta, m, \tq, \ell^{ps}, \ell^{pe},
u^{pe}, v^{pe}, D_{h, \eta}, X, p)$.

CASE uses tuning parameters ($\delta, m, \tq, \ell^{ps}, \ell^{pe},
u^{pe}, v^{pe}$). Earlier in this paper, we have briefly discussed how
to choose $(\delta, \tq)$. As for $m$, usually, a choice of $m = 2$ or
$3$ is sufficient unless the signals are relatively ``dense.'' The
choices of $(\ell^{ps}, \ell^{pe}, u^{pe}, v^{pe})$ are addressed in
Section~\ref{subsec:main}; see also Sections~\ref{subsec:Toeplitz}--\ref
{subsec:chang}.



\subsection{Computational complexity of CASE, comparison with
multivariate screening}
\label{subsec:compu}
The $\mathit{PS}$-step is closely related to the well-known method of
marginal screening and has a moderate computational complexity.

Marginal screening selects variables by thresholding the vector $d$
coordinate-wise.
The method is computationally fast, but it
neglects ``local'' graphical structures, and is thus ineffective.
For this reason, in many challenging problems, it is desirable to use
\textit{multivariate screening} methods which adapt to ``local''
graphical structures.

Fix $m > 1$. An $m$-variate $\chi^2$-screening procedure is one of such
desired methods.
The method screens all $k$-tuples of coordinates of $d$ using $\chi
^2$-tests, for all $k \leq m$, in an exhaustive (brute-force) fashion.
Seemingly, the method adapts to ``local'' graphical structures and
could be much more effective than marginal screening. However, such a
procedure has a computational cost of $O(p^m)$ [excluding the
computational cost for obtaining $X'Y$ from $(X, Y)$; same below] which
is usually not affordable when $p$ is large.

The main innovation of the $\mathit{PS}$-step is to use a graph-assisted
$m$-variate $\chi^2$-screening, which is both effective in variable
selection and efficient in computation.
In fact, the $\mathit{PS}$-step only screens $k$-tuples of coordinates of $d$ that
form a connected subgraph of ${\cal G}^*$, for all $k \leq m$.
Therefore, if ${\cal G}^*$ is $K$-sparse, then there are $\leq C p (e
K)^{m+1}$ connected
subgraphs of ${\cal G}^*$ with size $\leq m$; so if $K = K_p$ is no greater
than a
multi-$\log(p)$ term (see Definition~\ref{DefineLp}), then the
computational complexity of the $\mathit{PS}$-step is
only $O(p)$, up to a multi-$\log(p)$ term.

\renewcommand{\theexample}{\arabic{example}\normalfont{(c)}}
\setcounter{example}{1}
\begin{example}
We illustrate the difference between the above
three methods with the toy example in Example~\ref{ex2a}, where $p = 10$ and
the GOSD is displayed in Figure~\ref{Fig1}(a).
Suppose we choose $m = 3$. Marginal screening screens all $10$ single
nodes of the GOSD.
The brute-force $m$-variate screening screens all $k$-tuples of
indices, $1 \leq k \leq m$, with a total of ${p\choose1} + \cdots+
{p\choose m} = 175$ such $k$-tuples.
The $m$-variate screening in the $\mathit{PS}$-step only screens $k$-tuples that
are connected subgraphs of ${\cal G}^*$, for $1\leq k\leq m$, and in this
example, we only have $30$ such connected subgraphs.
\end{example}

The computational complexity of the $\mathit{PE}$-step consists two parts. The
first part is the complexity of obtaining all components of $\cU_p^*$,
which is $O(pK)$ and where $K$ is the maximum degree of ${\cal G}^+$; note
that for settings considered in this paper, $K=K_p^+$ does not exceed a
multi-$\log(p)$ term
[see Lemma B.2 in \citet{CASEsupp}]. 
The second part of the complexity comes from solving (\ref{obj of
minimization}), which
hinges on the maximal size of $\call$. In Lemma A.2 in \citet{CASEsupp}, 
we show that in a broad
context, the maximal size of $\call$ does not exceed a constant~$l_0$,
provided the thresholds $\tq$ are properly set. Numerical studies in
Section~\ref{sec:Simul} also support this point.
Therefore, the complexity in this part does not exceed $p\cdot3^{l_0}$.
As a result, the computational complexity of the $\mathit{PE}$-step is moderate.
Here, the bound $O(pK + p\cdot3^{l_0})$ is conservative; the actual
computational complexity is much smaller than this.

How does CASE perform? In Sections~\ref{subsec:minimax}--\ref{subsec:main},
we set up an asymptotic framework and show that CASE is
asymptotically minimax in terms of the Hamming distance over a wide
class of situations.
In Sections~\ref{subsec:Toeplitz}--\ref{subsec:chang}, we
apply CASE to the long-memory time series and the change-point model,
and elaborate the optimality of CASE in such models with
the so-called \textit{phase diagram}.

\subsection{Asymptotic rare and weak model} \label{subsec:minimax}
In this section, we add an asymptotic framework to the rare and weak
signal model $RW(\eps, \tau, \mu)$ introduced in Section~\ref{subsec:RW}. We use $p$ as the driving asymptotic parameter and tie
$(\eps, \tau)$
to $p$ through some fixed parameters.

In particular, we fix $\vartheta\in(0,1)$ and model the sparse
parameter $\eps$ by
%
\begin{equation}
\label{eps} \eps= \eps_p = p^{-\vartheta}.
\end{equation}
Note that as $p$ grows, the signal becomes increasingly sparse.
It turns out that the most interesting range of signal strength is
$\tau= O(\sqrt{\log(p)})$; see, for example, \citet{UPS}. For much
smaller $\tau$, successful recovery is impossible. For much
larger $\tau$, the problem is relatively easy. The critical value of
$\tau$ depends on $\vartheta$ in a complicate way. In light of this, we
fix $r>0$, and let
%
\begin{equation}
\label{tau} \tau= \tau_p = \sqrt{ 2r\log(p)}.
\end{equation}

At the same time, recalling that in $RW(\eps, \tau, \mu)$, we require
$\mu\in\Theta_p(\tau)$ so that $|\mu_i| \geq\tau$ for all $1
\leq i
\leq p$. Fixing $a > 1$, we now further restrict $\mu$ to the following
subset of $\Theta_p(\tau)$:
%
\begin{equation}
\label{muclass} \Theta_p^*(\tau_p, a)= \bigl\{ \mu\in
\Theta_p(\tau_p)\dvtx \tau_p\leq |\mu
_i|\leq a\tau_p, 1 \leq i \leq p \bigr\}.
\end{equation}

%
\begin{definition}
We call \eqref{eps}--\eqref{muclass} the asymptotic rare and weak model
$\mathit{ARW}(\vartheta, r, a, \mu)$.
\end{definition}

Requiring the strength of each signal $\leq a \tau_p$ is mainly for
technical reasons, and hopefully, such a constraint can be removed in
the near future.
From a practical point of view, since usually we do not have sufficient
information on $\mu$, we prefer to have a larger $a$: we hope that when
$a$ is properly large, $\Theta_p^*(\tau_p, a)$ is broad enough, so that
neither the optimal procedure nor the minimax risk
needs to adapt to $a$.

Toward this end, we impose some mild regularity conditions on $a$ and
the Gram matrix $G$. Let $g$ be the smallest integer such that
%
\begin{equation}
\label{Defineg} g \geq\max\bigl\{ (\vartheta+ r)^2/(2 \vartheta r), m
\bigr\}.
\end{equation}
For any $p \times p$ Gram matrix $G$ and $1 \leq k \leq p$, let
$\lambda
_k^*(G)$ be the minimum of the smallest eigenvalues of all $k\times k$
principle sub-matrices of $G$.
Introduce
%
\begin{equation}
\label{matrixclass2} \widetilde{{\cal M}}_p(c_0, g) = \bigl\{
\mbox{$G$ is a $p \times p$ Gram matrix, $\lambda_k^*(G) \geq
c_0$, $1 \leq k \leq g$}\bigr\}.
\end{equation}

For any two subsets $V_0$ and $V_1$ of $\{1, 2,\ldots, p\}$, consider
the optimization problem
\begin{eqnarray*}
&&\bigl(\theta^{(0)}_*(V_0, V_1; G),
\theta_*^{(1)}(V_0, V_1; G) \bigr) \\
&&\qquad= \margmin
\bigl\{ \bigl(\theta^{(1)} - \theta^{(0)}\bigr)' G
\bigl(\theta^{(1)} - \theta ^{(0)}\bigr) \bigr\},
\end{eqnarray*}
up to the constraints that $|\theta_i^{(k)}| \geq\tau_p$ if $i \in
V_k$ and $\theta_i^{(k)} = 0$ otherwise, where $k = 0, 1$, and that in
the special case of $V_0 = V_1$, the sign vectors of
$\theta^{(0)}$ and $\theta^{(1)}$ are unequal. Introduce
\[
a_g^*(G) = \max_{\{(V_0, V_1)\dvtx |V_0 \cup V_1| \leq g\}} \max\bigl\{ \bigl\|
\theta_*^{(0)}(V_0, V_1; G)\bigr\|_{\infty},\bigl \|
\theta_*^{(1)}(V_0, V_1; G) \bigr\|_{\infty}
\bigr\}.
\]
The following lemma is elementary, so we omit the proof.
%
\begin{lemma} \label{lemma:g}
For any $G \in\widetilde{{\cal M}}_p(c_0, g)$,
there is a constant $C = C(c_0, g) > 0$ such that $a_g^*(G) \leq C$.
\end{lemma}

In this paper, except for Section~\ref{subsec:chang} where we discuss
the change-point model, we assume
%
\begin{equation}
\label{Defineag} G \in\widetilde{{\cal M}}(c_0, g),\qquad a >
a^*_g(G).
\end{equation}
Under such conditions, $\Theta_p^*(\tau_p, a)$ is broad enough and the
minimax risk (to be introduced below)
does not depend on $a$. See Section~\ref{subsec:LB} for more discussion.

For any variable selection procedure $\hb$, we measure the performance
by the Hamming distance
\[
h_p(\hat{\beta}; \beta, G)= E \Biggl[ \sum
_{j=1}^p 1 \bigl\{ \operatorname {sgn}(\hat{
\beta}_j)\neq\operatorname{sgn}(\beta_j) \bigr\}
\Big| X, \beta \Biggr],
\]
where the expectation is taken with respect to $\hb$. Here, for any $p
\times1$ vector
$\xi$, $\sgn(\xi)$ denotes the sign vector [for any number $x$,
$\sgn
(x) = 1, 0, -1$ when $x < 0$, $x = 0$, and $x > 0$ correspondingly].

Under $\mathit{ARW}(\vartheta,r, a, \mu)$, $\beta= b \circ\mu$, so the overall
Hamming distance is
\[
H_p(\hat{\beta}; \varepsilon_p,\mu, G)=
E_{\eps_p} \bigl[ h_p(\hat {\beta}; \beta,G) | X
\bigr],
\]
where $E_{\eps_p}$ is the expectation with respect to the law of $b$.
Finally, the minimax Hamming distance under $\mathit{ARW}(\vartheta, r, a, \mu
)$ is
\[
\operatorname{Hamm}_p^*(\vartheta,r, a, G)=\inf_{\hat{\beta}} \sup
_{\mu\in
\Theta
_p^*(\tau_p, a)} H_p(\hat{\beta}; \varepsilon_p,
\mu, G).
\]
In next section, we will see that the minimax Hamming distance does not
depend on $a$ as long as (\ref{Defineag}) holds.

In many recent works, the \textit{probability of exact support
recovery} or \textit{oracle property}
is used to assess optimality; see, for example, \citet{FanLi,YuB,FXZ,Zou}.
However, when signals are rare and weak, exact support
recovery is usually impossible, and the Hamming distance is a more appropriate
criterion for assessing optimality. In comparison, study on the minimax
Hamming distance is not only mathematically more demanding but also
scientifically more relevant than that on the oracle property.

\subsection{Lower bound for the minimax Hamming distance} \label{subsec:LB}
We view the\break (global) Hamming distance as the aggregation of
``local'' Hamming errors. To construct a lower bound for the (global) minimax
Hamming distance, the key is to construct lower bounds
for ``local'' Hamming errors. Fix
$1 \leq j \leq p$. The ``local'' Hamming error at index $j$
is the risk we make among the neighboring indices of $j$ in GOSD, say,
$\{k\dvtx d(j,k) \leq g\}$,
where $g$ is as in (\ref{Defineg}) and $d(j,k)$ is the geodesic
distance between $j$ and $k$
in the GOSD.
The lower bound for such a ``local'' Hamming error is characterized by
an exponent
$\rho_j^*$, which we now introduce.

For any subset $V \subset\{1, 2,\ldots, p\}$, let $I_V$ be the $p
\times1$ vector such that the $j$th coordinate is $1$ if $j \in V$ and
$0$ otherwise. Fixing two subsets $V_0$ and $V_1$ of $\{1, 2,\ldots,
p\}$, we introduce
\begin{equation}
\label{Definevarpi} \varpi^*(V_0, V_1) =
\tau_p^{-2} \min_{\theta^{(0)}, \theta^{(1)} } \bigl\{\bigl(
\theta^{(1)} - \theta^{(0)}\bigr)' G \bigl(
\theta^{(1)} - \theta^{(0)}\bigr) \bigr\},
\end{equation}
subject to $ \{\theta^{(k)} = I_{V_k} \circ\mu^{(k)}\dvtx \mu^{(k)} \in
\Theta_p^*(\tau_p, a), k = 0, 1, \sgn(\theta^{(0)})\neq\sgn
(\theta
^{(1)})\}$,
and let
%
\begin{eqnarray}
\label{defrho} \rho(V_0, V_1)&=& \max\bigl\{|V_0|,
|V_1|\bigr\}\vartheta
\nonumber
\\[-8pt]
\\[-8pt]
\nonumber
&&{}+ \frac{1}{4} \biggl[ \biggl( \sqrt{
\varpi^*(V_0, V_1)r}- \frac{|(|V_1|-|V_0|)|\vartheta}{\sqrt
{\varpi^*(V_0, V_1) r}} \biggr)_+
\biggr]^2.
\end{eqnarray}
The exponent $\rho_j^* = \rho_j^*(\vartheta, r, a, G)$ is defined by
%
\begin{equation}
\label{def of rho_j} \rho_j^*(\vartheta, r, a, G)= \min
_{(V_0, V_1)\dvtx j\in V_0\cup V_1} \rho (V_0, V_1).
\end{equation}
The notation $L_p$ is frequently used in this paper.
%
\begin{definition}\label{DefineLp}
$L_p$, as a positive sequence indexed by $p$, is called a~multi-$\log
(p)$ term\vspace*{1pt} if for any fixed $\delta>0$, $\lim_{p \goto\infty}
L_pp^{\delta} = \infty$ and\break $\lim_{p\goto\infty}L_pp^{-\delta}=0$.
\end{definition}
It can be shown that $L_p p^{-\rho_j^*}$ provides a lower bound for the
``local''
minimax Hamming distance at index $j$, and that
when (\ref{Defineag}) holds, $\rho_j^*(\vartheta, r, a, G)$ does not
depend on $a$; see Lemma~16 in \citet{JZ11} for details. In the
remaining part of the paper, we will write it as $\rho_j^*(\vartheta
,r,G)$ for short.

At the same time, in order for the aggregation of all lower bounds for
``local'' Hamming errors to give a lower bound
for the ``global'' Hamming distance, we need to introduce \textit{graph
of least favorables} (GOLF).
Toward this end, recalling $g$ and $\rho(V_0, V_1)$ as in (\ref
{Defineg}) and \eqref{defrho}, respectively, let
\[
\bigl(V_{0j}^*, V_{1j}^*\bigr) =
\mathop{\margmin\limits_{\{(V_0, V_1)\dvtx j \in V_0 \cup V_1, |V_0
\cup V_1| \leq g\}}}
\rho(V_0, V_1),
\]
and when there is a tie, pick the one that appears first lexicographically.
We can think $(V_{0j}^*, V_{1j}^*)$ as the ``least favorable''
configuration at index $j$.
%
\begin{definition}\label{DefineGOLF}
GOLF is the graph ${\cal G}^{\diamond} = (V, E)$ where $V = \{1, 2,\ldots,
p\}$ and there is an edge between $j$ and $k$ if and only if $(V_{0j}^*
\cup V_{1j}^*) \cap(V_{0k}^* \cup V_{1k}^*) \neq\varnothing$.
\end{definition}
The following theorem is similar to Theorem~14 in \citet{JZ11}, so we
omit the proof.
%
\begin{thmm} \label{thmm:LB}
Suppose (\ref{Defineag}) holds so that $\rho_j^*(\vartheta, r, a, G)$
does not depend on the parameter $a$ for sufficiently large $p$.
As $p\to\infty$, $\operatorname{Hamm}^*_p(\vartheta, r, a,\break G)\geq L_p [d_p({\cal G}
^{\diamond})]^{-1}\sum_{j=1}^p p^{-\rho_j^*(\vartheta, r, G)}$, where
$d_p({\cal G}^{\diamond})$ is the maximum degree of all nodes in
${\cal G}
^{\diamond}$.
\end{thmm}
In many examples, including those of primary interest of this paper,
%
\begin{equation}
\label{golf} d_p\bigl({\cal G}^{\diamond}\bigr) \leq
L_p.
\end{equation}
In such cases, we have the following lower bound:
%
\begin{equation}
\label{LB} \operatorname{Hamm}^*_p(\vartheta, r, a, G)\geq L_p \sum
_{j=1}^p p^{-\rho
_j^*(\vartheta
, r, G)}.
\end{equation}

\subsection{Upper bound and optimality of CASE} \label{subsec:main}
In this section, we show that in a broad context, provided the tuning
parameters are properly set,
CASE achieves the lower bound prescribed in Theorem~\ref{thmm:LB}, up
to some $L_p$ terms. Therefore, the lower bound in Theorem~\ref
{thmm:LB} is tight, and CASE achieves the optimal rate of convergence.

For a given $\gamma> 0$, we focus on linear models with the Gram
matrix from
\[
{\cal M}_p^*(\gamma, g, c_0, A_1) =
\widetilde{{\cal M}}_p(c_0, g) \cap\cM_p(
\gamma, A_1),
\]
where we recall that the two terms on the right-hand side are defined
in \eqref{matrixclass} and~\eqref{matrixclass2}, respectively. The
following lemma is proved in Section B in \citet{CASEsupp}. 

\begin{lemma} \label{lem:GOLFdegree}
For $G\in\cM^*_p(\gamma,g,c_0, A_1)$, the maximum degree of nodes in
GOLF satisfies $d_p(\cG^{\diamond})\leq L_p$.
\end{lemma}

Combining Lemma~\ref{lem:GOLFdegree} with Theorem~\ref
{thmm:LB}, the lower bound \eqref{LB} holds.

For any linear filter $D = D_{h, \eta}$, let
$\varphi_{\eta}(z) = 1 + \eta_1 z +\cdots+ \eta_h z^h$
be the so-called \textit{characterization polynomial}.
We need some regularity conditions:
\begin{itemize}
\item\textit{Regularization Condition \textup{A} \textup{(}RCA\textup{)}}. For any root $z_0$ of
$\varphi_{\eta}(z)$, $|z_0| \geq1$.
\item\textit{Regularization Condition \textup{B} \textup{(}RCB\textup{)}}.
There are constants $\kappa> 0$ and $c_1>0$ such that $\lambda_k^*(D G
D') \geq c_1 k^{-\kappa}$ ($\lambda_k^*$ is as in Section~\ref{subsec:LB}).
\end{itemize}
For many well-known linear filters such as adjacent differences,
seasonal differences, etc., RCA is satisfied.
Also, RCB is only a mild condition since $\kappa$ can be any positive
number. For example, RCB holds in the change-point model and
long-memory time series model with certain $D$ matrices.
In general, $\kappa$ is not $0$ because when $DG$ is sparse, $DGD'$ is
very likely to be
approximately singular, and the associated value of $\lambda_k^*$ can
be small when $k$ is large. This
is true even for very simple~$G$ [e.g., $G = I_p$, $D = D_{1, \eta}$
and $\eta= (1, -1)'$].

At the same time, these conditions can be further relaxed. For example,
for the change-point problem, the Gram matrix has barely any
off-diagonal decay,
and does not belong to ${\cal M}_p^*$. Nevertheless, with slight
modification in the procedure, the main results continue to hold.

CASE uses tuning parameters $(\delta, m, \tq, \ell^{ps}, \ell^{pe},
u^{pe}, v^{pe})$.
The choice of $\delta$ is flexible, and we usually set $\delta=
1/\log(p)$.
For the main theorem below, we treat $m$ as given. In practice,
taking $m$ to be a small integer (say, $\leq3$) is usually sufficient, unless
the signals are relatively dense (say, $\vartheta< 1/4$).
The choice of $\ell^{ps}$ and $\ell^{pe}$ are also relatively flexible,
and letting $\ell^{ps}$ be a sufficiently large constant and
$\ell^{pe}$ be $(\log(p))^{\nu}$ for some constant $\nu< (1 -
1/\alpha
)/(\kappa+ 1/2)$ is sufficient, where $\alpha$ is as in Definition~\ref
{def:sparse}, and $\kappa$ is as in RCB.

At the same time, in principle, the optimal choices of
$(u^{pe}, v^{pe})$ are
%
\begin{equation}
\label{paracond} u^{pe}=\sqrt{2\vartheta\log p},\qquad v^{pe}=
\sqrt{2r\log p},
\end{equation}
which depend on the underlying parameters $(\vartheta, r)$ that are
unknown to us.
Despite this, our numeric studies in Section~\ref{sec:Simul} suggest
that the choices of $(u^{pe}, v^{pe})$ are relatively flexible; see
Sections~\ref{sec:Simul}--\ref{sec:Discu}
for more discussions.

Last, we discuss how to choose $\tq= \{t(\hat{F}, \hat{N})\dvtx \mbox
{$(\hat{F}, \hat{N})$ are defined as in the}$ $\mbox{$\mathit{PS}$-step}\}$. Let
$t(\hat{F}, \hat{N}) = 2 q \log(p)$, where $q > 0$ is a constant.
It turns out that the main result (Theorem~\ref{thmm:main} below) holds
as long as
%
\begin{equation}
\label{tq} q_0 \leq q \leq q^*(\hat{F}, \hat{N}),
\end{equation}
where $q_0>0$ is an appropriately small constant, and for any subsets
$(F, N)$,
%
\begin{eqnarray}\qquad
\label{Defineq*}&& q^*(F,N)
\nonumber
\\[-8pt]
\\[-8pt]
\nonumber
&&\qquad= \max \bigl\lbrace q\dvtx \bigl(|F|+|N|\bigr)\vartheta+ \bigl[
\bigl(\sqrt{\tilde {\omega }(F,N)r}-\sqrt{q|F|}\bigr)_+\bigr]^2\geq
\psi(F,N) \bigr\rbrace;
\end{eqnarray}
here,
%
\begin{eqnarray}\quad
\label{psi} \psi(F,N) &= &\frac{(|F|+2|N|)\vartheta}{2}
\nonumber
\\[-8pt]
\\[-8pt]
\nonumber
&&{}+ \cases{ %
 \displaystyle\frac{1}{4}\omega(F,N)r, \qquad |F| \mbox{ is even},
\vspace*{2pt}\cr
\displaystyle\frac{\vartheta}{2} + \frac{1}{4}\bigl[\bigl( \sqrt{\omega(F,N)r}- \vartheta
/\sqrt {\omega(F,N)r} \bigr)_+\bigr]^2,\vspace*{2pt}\cr
\hspace*{78pt}|F| \mbox{ is odd},}
\end{eqnarray}
with
%
\begin{equation}
\label{omega(F,N)} \omega(F,N)= \min_{\xi\in\mathbb{R}^{|F|}\dvtx |\xi_i|\geq1} \xi'
\bigl[G^{F,F}- G^{F,N}\bigl(G^{N,N}
\bigr)^{-1}G^{N,F}\bigr]\xi
\end{equation}
and
%
\begin{equation}
\label{tomega(F,N)} \tilde{\omega}(F, N) = \min_{\xi\in\mathbb{R}^{|F|}\dvtx |\xi_i|\geq
1} \xi
'\bigl[Q_{F,F}- Q_{F,N}(Q_{N,N})^{-1}Q_{N,F}
\bigr]\xi,
\end{equation}
where $Q_{F,N}= (B^{\call^{ps},F})'(H^{\call^{ps},\call^{ps}})^{-1}
(B^{\call^{ps},N})$ with $\call=F\cup N$, and $Q_{N,F}$, $Q_{F,F}$ and
$Q_{N,N}$ are defined similarly. Compared to \eqref{def of W and
Q(ps)}, we see that $Q_{F,N}$, $Q_{F,N}$, $Q_{N,F}$ and $Q_{N,N}$ are
all submatrices of $Q$. Hence, $\tilde{\omega}(F,N)$ can be viewed as a
counterpart of $\omega(F,N)$ by replacing the submatrices of $G^{\call
,\call}$ by the corresponding ones of $Q$.

From a practical point of view, there is a trade-off in choosing $q$: a
larger $q$ would increase the number of falsely selected variables in
the $\mathit{PS}$-step, but would also reduce the computational cost in the
$\mathit{PE}$-step. The following is a convenient choice which we recommend in
this paper:
%
\begin{equation}
\label{constant critical value} t(\hat{F},\hat{N}) = 2\tilde{q}|\hat{F}|\log(p),
\end{equation}
where $0 < \tilde{q} < c_0 r /4$ is a constant, and $c_0$ is as in
${\cal M}_p^*(\gamma, g, c_0, A_1)$.

We are now ready for the main result of this paper.
%
\begin{thmm} \label{thmm:main}
Suppose that for sufficiently large $p$, $G \in{\cal M}_p^*(\gamma, g,
c_0, A_1)$, $D_{h, \eta} G \in\cM_p(\alpha, A_0)$ with $\alpha>1$ and
that RCA-RCB hold.
Consider $\hat{\beta}^{\mathrm{case}} = \hat{\beta}^{\mathrm{case}}(Y; \delta, m,
\tq,
\ell^{ps}, \ell^{pe}, u^{pe}, v^{pe}, D_{h, \eta}, X, p)$
with the tuning parameters specified above.
Then as $p \goto\infty$,
%
\begin{eqnarray}
\label{main UB} &&\sup_{\mu\in\Theta_p^*(\tau_p, a)} H_p\bigl(\hat{
\beta}^{\mathrm{case}}; \varepsilon _p,\mu, G\bigr)
\nonumber
\\[-8pt]
\\[-8pt]
\nonumber
&&\qquad\leq L_p
\Biggl[p^{1-(m+1) \vartheta} + \sum_{j=1}^p
p^{-\rho
_j^*(\vartheta,r, G)} \Biggr] + o(1).
\end{eqnarray}
\end{thmm}
Combine Lemma~\ref{lem:GOLFdegree} and Theorem~\ref{thmm:main}. Given
the parameter $m$ is appropriately large, both the upper bound and the
lower bound are tight, and CASE achieves the optimal rate of
convergence prescribed by
%
\begin{equation}
\label{simple UB} \hamm_p^*(\vartheta, r, a, G) = L_p \sum
_{j=1}^p p^{-\rho
_j^*(\vartheta
, r, G)} + o(1).
\end{equation}
Theorem~\ref{thmm:main} is proved in Section A in \citet{CASEsupp}, 
where we explain the key idea behind the procedure, as well as the
selection of the tuning parameters.

\subsection{Application to the long-memory time series model} \label
{subsec:Toeplitz}
The long-me\-mory time series model in Section~\ref{sec:intro} can be
written as a regression model,
\[
Y = X\beta+ z, \qquad z\sim N(0, I_n),
\]
where the Gram matrix $G$ is asymptotically Toeplitz and has slow
off-diagonal decays. Without loss of generality, we consider the
following idealized case where $G$ is an exact Toeplitz matrix
generated by a spectral density $f$:
%
\begin{equation}
\label{lmG} G(i,j) = \frac{1}{2\pi} \int_{-\pi}^{\pi}
\cos\bigl(|i - j| \omega\bigr) f(\omega ) \,d \omega,\qquad 1 \leq i, j \leq p.
\end{equation}
In the literature [\citet{ChenHur,Moulines}], the spectral density for a
long-memory process is usually characterized as
%
\begin{equation}
\label{lmmodel} f(\omega) = \bigl|1 - e^{\sqrt{-1} \omega}\bigr|^{-2\phi} f^*(\omega),
\end{equation}
where $\phi\in(0,1/2)$ is the long-memory parameter, $f^*(\omega)$ is
a positive symmetric function that is continuous
on $[-\pi, \pi]$ and is twice differentiable except at $\omega= 0$.

In this model, the Gram matrix is nonsparse, but it is sparsifiable.
To see the point, let
$\eta= (1, -1)'$ and let $D = D_{1, \eta}$ be the first-order adjacent
row-differencing. On one hand, since the spectral density $f$ is
singular at the origin, it follows from the Fourier analysis that
$|G(i,j)| \geq C (1 + |i - j|)^{-(1 - 2\phi)}$,
and hence $G$ is nonsparse. On the other hand,
it is seen that
\[
B(i,j) = \sqrt{-1} \int_{|j-i|}^{|j-i|+1} \widehat{\omega
f(\omega )}(\lambda)\,d\lambda,
\]
where we recall that $B = D G$ and note that $\hat{g}$ denotes the
Fourier transform of $g$.
Compared to $f(\omega)$, $\omega f(\omega)$ is nonsingular at the
origin. Additionally, it is seen that $B \in\mathcal{M}_p(2-2\phi,
A)$, where $2 - 2 \phi> 1$, so $B$ is sparse (a similar claim applies
to $H = D G D'$). This shows that $G$ is sparsifiable by adjacent
row-differencing.

In this example, there is a function $\rho_{\mathrm{lts}}^*(\vartheta, r; f)$
that only depends on $(\vartheta, r, f)$ such that
\[
\max_{\{ j\dvtx \log(p) \leq j \leq p - \log(p) \}} \bigl\{\bigl | \rho _j^*(\vartheta , r,
G) - \rho_{\mathrm{lts}}^*(\vartheta, r; f)\bigr| \bigr\} \goto0\qquad \mbox{as $p \goto
\infty$},
\]
where the subscript ``lts'' stands for long-memory time series.
The following theorem can be derived from Theorem~\ref{thmm:main}, and
is proved in Section B in \citet{CASEsupp}. 

\begin{thmm} \label{thmm:Toeplitz}
For a long-memory time series model where $|(f^*)''(\omega)| \leq C
|\omega|^{-2}$, the minimax Hamming distance then satisfies $\hamm
_p^*(\vartheta, r, G) =  L_p p^{1- \rho_{\mathrm{lts}}^*(\vartheta, r; f)}$.
If we apply CASE by letting $(m + 1) \vartheta> \rho
_{\mathrm{lts}}^*(\vartheta
, r; f)$,
$\eta= (1, -1)'$, and the tuning parameters be set as in Section~\ref{subsec:main},
then
\[
\sup_{\mu\in\Theta_p^*(\tau_p, a)} H_p\bigl(\hat{\beta}^{\mathrm{case}};
\eps _p, \mu, G\bigr) \leq L_p p^{1- \rho_{\mathrm{lts}}^*(\vartheta, r; f)} + o(1).
\]
\end{thmm}

Theorem~\ref{thmm:Toeplitz} can be interpreted by the so-called
\textit
{phase diagram}.
Phase diagram is a way to visualize settings where the signals
are so rare and weak that successful variable selection is simply
impossible [\citet{UPS}].
In detail, for a spectral density $f$ and $\vartheta\in(0,1)$, let
$r_{\mathrm{lts}}^*(\vartheta) = r_{\mathrm{lts}}^*(\vartheta; f)$
be the unique solution of $\rho_{\mathrm{lts}}^*(\vartheta, r; f) = 1$.
Note that $r = r^*_{\mathrm{lts}}(\vartheta)$ characterizes the minimum signal
strength required for exact support recovery with high probability.
The following proposition is proved in Section B in \citet{CASEsupp}. 
%
\begin{lemma} \label{lem:Tpboundary}
Under the conditions of Theorem~\ref{thmm:Toeplitz},
if $(f^*)''(0)$ exists,
then $r^*_{\mathrm{lts}}(\vartheta; f)$ is a decreasing function in $\vartheta$,
with limits $1$ and $\frac{2}{\pi} \int_{-\pi}^{\pi} f^{-1}(\omega
) \,d
\omega$ as $\vartheta\goto1$ and $\vartheta\goto0$, respectively.
\end{lemma}
%


Call the two-dimensional space $\{(\vartheta, r)\dvtx 0 <
\vartheta< 1, r > 0\}$ the \textit{phase space}. Interestingly,
there is a partition of the phase space as follows.
\begin{itemize}
\item\textit{Region of no recovery $\{(\vartheta,r)\dvtx 0 < r
<\vartheta, 0 < \vartheta< 1\}$}. In this region, the minimax
Hamming distance $\gtrsim p \eps_p$, where $p \eps_p$ is approximately
the number of signals. In this region, the signals are too rare and
weak and successful variable selection is impossible.
\item\textit{Region of almost full recovery $\{(\vartheta,r):
\vartheta< r < r_{\mathrm{lts}}^*(\vartheta; f), 0 < \vartheta< 1 \}$}. In
this region, the minimax Hamming distance is much larger than $1$ but
much smaller than $p \eps_p$. Therefore, the optimal procedure can
recover most of the signals but not all of them.
\item\textit{Region of exact recovery $\{(\vartheta,r): r >
r_{\mathrm{lts}}^*(\vartheta; f), 0 < \vartheta< 1 \}$}.
In this region, the
minimax Hamming distance is $o(1)$. Therefore, the optimal procedure
recovers all signals with probability $\approx1$.
\end{itemize}
Because of the partition of the phase space, we call this the \textit
{phase diagram}.

From time to time, we wish to have a more explicit formula for the rate
$\rho_{\mathrm{lts}}^*(\vartheta, r; f)$ and the critical value
$r_{\mathrm{lts}}^*(\vartheta; f)$.
In general, this is a hard problem, but both quantities can be
computed numerically when $f$ is given. In Figure~\ref{fig:toeplitz},
we display the phase diagrams for the \textit{autoregressive
fractionally integrated moving average} process (FARIMA) with
parameters $(0, \phi, 0)$ [\citet{FanTimeS}], where
%
\begin{equation}
\label{f in FARIMA} f^*(\omega) = \frac{\Gamma^2(1-\phi)}{\Gamma(1-2\phi)}.
\end{equation}
Take $\phi= 0.35, 0.25$, for example, $r_{\mathrm{lts}}^*(\vartheta; f) \approx
7.14, 5.08$ for small
$\vartheta$.


\begin{figure}[b]

\includegraphics{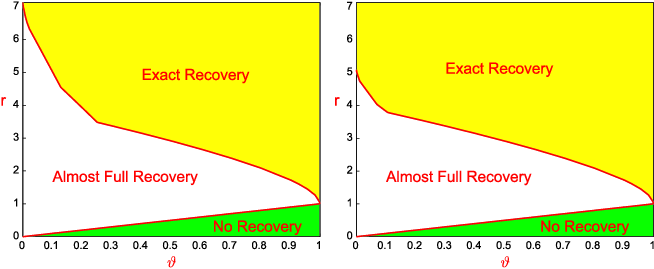}

\caption{Phase diagrams corresponding to the $\operatorname{FARIMA}(0,\phi,0)$
process. Left: $\phi= 0.35$. Right: $\phi= 0.25$.} \label
{fig:toeplitz}
\end{figure}

\subsection{Application to the change-point model} \label{subsec:chang}
The change-point model in the \hyperref[sec:intro]{Introduction} can be viewed as a special
case of model (\ref{model}), where $\beta$ is as in (\ref{sparse signal
model on beta}), and the Gram matrix satisfies
%
\begin{equation}
G(i,j)= \min\{i,j\}, \qquad 1\leq i,j\leq p. \label{Gchangpt}
\end{equation}
For technical reasons, it is more convenient \textit{not} to normalize
the diagonals of $G$ to~$1$.

The change-point model can be viewed as an ``extreme'' case of what is
studied in this paper. On one hand,
the Gram matrix $G$ is ``ill-posed,'' and each row of $G$ does not
satisfy the condition of off-diagonal decay in Theorem~\ref{thmm:main}.
On the other hand, $G$ has a very special structure which can be
largely exploited. In fact, if we sparsify $G$ using the linear filter
$D = D_{2, \eta}$, where $\eta= (1, -2, 1)'$, it is seen that $B = DG
= I_p$, and $H = D G D'$ is a tri-diagonal matrix with $H(i,j)= 2\cdot
1\{i=j\}- 1\{|i-j|=1\}- 1\{i=j=p\}$, which are very simple matrices.
For these reasons, we modify the CASE as follows:
\begin{itemize}
\item Due to the simple structure of $B$, we do not need patching in
the $\mathit{PS}$-step (i.e., $\ell^{ps} = 0$).
\item For the same reason, the choices of thresholds $t(\hat{F}, \hat
{N})$ are more flexible than before, and taking $t(\hat{F},\hat{N}) =
2q\log(p)$ for a proper constant $q > 0$ works.
\item Since $H$ is ``extreme'' (the smallest eigenvalue tends to $0$ as
$p \goto\infty$), we have to modify the $\mathit{PE}$-step carefully.
\end{itemize}

In detail, the $\mathit{PE}$-step for the change-point model is as follows.
Given $\ell^{pe}$, let ${\cal G}^+$ be as in Definition~\ref{DefineGplus}.
Recall that ${\cal U}_p^*$ denotes the set of all retained indices at
the end of the $\mathit{PS}$-step. We view ${\cal U}_p^*$ as a subgraph of
${\cal G}
^+$, and let $\call\lhd{\cal U}_p^*$ be one of its components. The
goal is to
split $\call$ into $N$ different subsets
\[
\call= \call^{(1)} \cup\cdots\cup\call^{(N)},
\]
and for each subset $\call^{(k)}$, $1 \leq k \leq N$, we construct a
patched set $\call^{(k), pe}$. We then estimate $\beta^{\call^{(k)}}$
separately using \eqref{obj of minimization}. Putting
$\beta^{\call^{(k)}}$ together gives our estimate of~$\beta^{\call}$.

The subsets $\{(\call^{(k)}, \call^{(k), pe})\}_{k = 1}^N$ are
recursively constructed as follows. Denote $l = |\call|$, $M = (\ell
^{pe}/2)^{1/(l + 1)}$, and write
\[
\call= \{j_1, j_2,\ldots, j_l\},\qquad
j_1 < j_2 < \cdots< j_l.
\]
First, letting $k_1$ be the largest index such that
$j_{k_1}-j_{k_1-1}> \ell^{pe}/M$, define
\[
\call^{(1)} = \{j_{k_1},\ldots,j_l\}\quad
\mbox{and}\quad \call^{(1),pe} = \bigl\{j_{k_1} - \ell^{pe}/(2M),
\ldots, j_l+ \ell^{pe}/2 \bigr\}.
\]
Next, letting $k_2 < k_1$ be the largest index such that
$j_{k_2}-j_{k_2-1}> \ell^{pe}/M^2$, define
\[
\call^{(2)} = \{j_{k_2},\ldots, j_{k_1} \},\qquad
\call^{(2), pe} =\bigl\{ j_{k_2}-\ell^{pe}/
\bigl(2M^2\bigr),\ldots, j_{k_1}+\ell^{pe}/(2M)
\bigr\}.
\]
Continue this process until for some $N$, $1 \leq N \leq l$, $k_N = 1$.
In this construction, for each $1 \leq k \leq N$, if we arrange all the
nodes of $\call^{(k),pe}$ in the ascending order, then the number of
nodes in front of $\call^{(k)}$ is significantly smaller than the
number of nodes behind $\call^{(k)}$.

In practice, we introduce a suboptimal but much simpler
patching approach as follows.
Fix a component $\call=\{j_1,\ldots,j_l\}$ of $\cU_p^*$. In this
approach, instead of splitting it into smaller sets and patching them
separately as in the previous approach, we patch the whole set $\call$ by
%
\begin{equation}
\label{easypatch} \call^{pe} = \bigl\{i\dvtx j_1-
\ell^{pe}/4 < i < j_l + 3\ell^{pe}/4 \bigr\},
\end{equation}
and estimate $\beta^{\call}$ using \eqref{obj of minimization}.
Our numeric studies show that two approaches have comparable performances.

Define
%
\begin{equation}
\label{lower bound (rate)} \rho^*_{\mathrm{cp}}(\vartheta, r)= \cases{ %
 \vartheta+r/4, &\quad $r/\vartheta\leq6+2\sqrt{10},$
\vspace*{2pt}\cr
3\vartheta+(r/2-\vartheta)^2/(2r), &\quad $r/\vartheta> 6+2\sqrt{10},$}
\end{equation}
where ``cp'' stands for change-point.
Choose the tuning parameters of CASE such that
%
\begin{equation}\quad
\label{paracond2} \ell^{pe}=2\log(p), \qquad u^{pe}=\sqrt{2\vartheta
\log(p)}\quad \mbox{and}\quad v^{pe}= \sqrt{2r\log(p)},
\end{equation}
that $(m+1)\vartheta\geq\rho^*_{\mathrm{cp}}(\vartheta,r)$, and that $0<q<
\frac{r}{4} (\sqrt{2}-1)^2$ [recall that we take $t(\hat{F}, \hat
{N}) =
2 q \log(p)$ for all $(\hat{F}, \hat{N})$ in the change-point setting].
Note that the choice of $\ell^{pe}$ is different from that in Section~\ref{subsec:ags}.
The main result in this section is the following theorem which is
proved in Section B in \citet{CASEsupp}. 
%
\begin{thmm} \label{thmm:changpt}
For the change-point model, the minimax Hamming distance satisfies
$\hamm_p^*(\vartheta, r, G) = L_p p^{1 - \rho^*_{\mathrm{cp}}(\vartheta, r)}$.
Furthermore, $\hat{\beta}^{\mathrm{case}}$ with the tuning parameters specified
above satisfies
\[
\sup_{\mu\in\Theta_p^*(\tau_p, a)} H_p\bigl(\hat{\beta}^{\mathrm{case}};
\varepsilon_p, \mu, G\bigr)\leq L_p p^{1 - \rho^*_{\mathrm{cp}}(\vartheta, r)} +
o(1).
\]
\end{thmm}

It is noteworthy that the exponent $\rho_{\mathrm{cp}}^*(\vartheta, r)$ has a
phase change depending on the ratios of
$r / \vartheta$. The insight is, when $r/\vartheta< 6 + 2 \sqrt{10}$,
the minimax Hamming distance is dominated by
the Hamming errors we make in distinguishing between an isolated
change-point and a pair of adjacent change-points, and when $r /
\vartheta> 6 + 2 \sqrt{10}$, the minimax Hamming distance is dominated
by the Hamming errors of distinguishing the case of consecutive
change-point triplets (say, change-points at $\{j-1, j, j-1\}$) from
the case where we do not have a change-point in the middle of the
triplets (i.e., the change-points are only at $\{j-1, j+1\}$).

\begin{figure}

\includegraphics{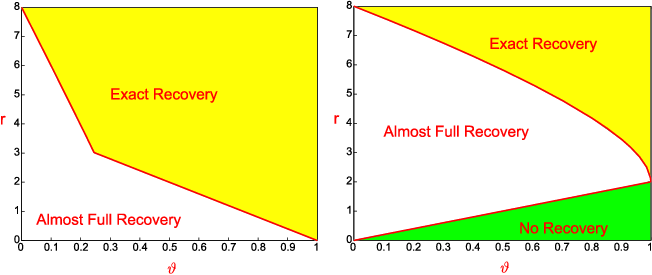}

\caption{Phase diagrams corresponding to the change-point model. Left:
CASE; the boundary is decided by $(4-10\vartheta)+2\sqrt
{(2-5\vartheta
)^2-\vartheta^2}$ (left part)
and $4(1-\vartheta)$ (right part). Right: hard thresholding; the upper
boundary is decided by $2(1+\sqrt{1-\vartheta})^2$
and the lower boundary is decided by~$2\vartheta$.} \label{fig:changpt}
\end{figure}
%

Similarly, the main results on the change-point problem can be
visualized with the phase diagram in
Figure~\ref{fig:changpt}. An interesting point is that it is possible
to have almost full recovery even when the signal
strength parameter $\tau_p$ is as small as $o(\sqrt{2 \log(p)})$. See
the proof of Theorem~\ref{thmm:changpt} for details.

Alternatively, one may apply the following approach to the change-point
problem. Treat the liner change-point model as a regression model $Y =
X \beta+ z$ as in Section~\ref{sec:intro} (page~\pageref{p1}), and let $W =
(X'X)^{-1} X'Y$ be the least-squares estimate. It is seen that
$W \sim N(\beta, \Sigma)$,
where we note that $\Sigma= (X'X)^{-1}$ is tridiagonal and coincides
with $H$. In this simple setting, a natural approach is to apply a
coordinate-wise thresholding $\hb^{\mathrm{thresh}}_j= W_j 1\{|W_j| > t\}$ to
locate the signals. But this
neglects the covariance of $W$ in detecting the locations of the
signals and is not optimal even with the ideal choice of thresholding
parameter $t_0$,
since the corresponding risk satisfies
\[
\sup_{\{\mu\in\Theta_p^*(\tau_p, a) \} } H_p\bigl(\hb^{\mathrm{thresh}}(t_0);
\varepsilon_p, \mu, G\bigr) = L_p p^{1-(r/2 + \vartheta)^2/(2r)}.
\]
The proof of this is elementary and omitted.
The phase diagram of this method is displayed in Figure~\ref{fig:changpt},
right panel, which suggests the method is nonoptimal.

Other popular methods in locating multiple change-points
include the global methods [\citet{HarLevy10,Olshen04,Tibshirani,YaoAn}] and local methods [e.g., SaRa in
\citet{SaRa}].
The global methods are usually computationally expensive and can hardly
be optimal due to the strong correlation nature of this problem. Our
procedure is related to the local methods
but is different in important ways. Our method exploits the graphical
structures and uses the GOSD to guide both the
screening and cleaning, but SaRa does not utilize the graphical
structures and can be shown to be nonoptimal.


To conclude the section, we remark that the change-point model
constitutes a special case of the settings we discuss in the paper,
where setting some of the tuning parameters is more convenient than in
the general case. First, for the change-point model, we can simply set
$\delta= 0$ and $\ell^{ps} = 0$. Second, there is an easy-to-compute
preliminary estimator available. On the other hand, the performance of
CASE is substantially better than the other methods in many situations.
We believe that CASE is potentially a very useful method in practice
for the change-point problem.

\section{Simulations} \label{sec:Simul}
We conducted a small-scale numeric study where we compare CASE and
several popular
approaches. The study contains two parts, Section~\ref{subsec:simuchang} and Section~\ref{subsec:simutoep}, where we
investigate the change-point model and the long-memory time series
model, respectively.

In this section, $s_p = p \eps_p$ for convenience. The core tuning
parameters for CASE are $(\tq, u^{pe}, v^{pe}, \ell^{ps}, \ell^{pe})$.
We streamline these tuning parameters in a way so they only depend on
two tuning parameters $(s_p, \tau_p)$ (calibrating the sparsity and the
minimum signal strength, resp.). Therefore, essentially, CASE only uses two
tuning parameters. Our experiments show that the performance of CASE is
relatively insensitive to these two tuning
parameters. Furthermore, these two tuning parameters can be set in a
data driven fashion, especially in the change-point model. See details below.

We set $m = 2$ so that in the screening stage of CASE, bivariate
screening is the highest order screening we use. At least for examples
considered here, using a higher-order screening does not have a
significant improvement.
For long-memory time series, we need a regularization parameter $\delta
$ (but we do not need it for the change-point model).
The guideline for choosing $\delta$ is to make sure the maximum degree
of GOSD is $15$ (say) or smaller. In this section, we choose
$\delta= 2.5/\log(p)$. The maximum degree of GOSD is much higher if we
choose a much smaller $\delta$; in this case, CASE has similar
performance, but is computationally much slower.

\subsection{Change-point model} \label{subsec:simuchang}
In this section, we use model \eqref{changept} to investigate the
performance of CASE in identifying multiple change-points. For a given
set of parameters $(p, \vartheta, r, a)$, we set $\eps_p =
p^{-\vartheta
}$ and $\tau_p = \sqrt{2 r \log(p)}$. First, we generate a $(p-1)
\times1$ vector $\beta$ by $\beta_j \stackrel{\mathrm{i.i.d.}}{\sim} (1-\eps
_p) \nu
_0 + \frac{\eps_p}{2} U(\tau_p, a \tau_p) + \frac{\eps_p}{2} U(-
a \tau
_p, -\tau_p)$, where $U(s,t)$ is the uniform distribution over $[s,t]$
[when $s = t$, $U(s, t)$ represents the point mass at $s$].
Next, we construct the
mean vector $\theta$ in model (\ref{changept}) by
$\theta_j=\theta_{j-1}+\beta_{j-1}$, $2 \leq j\leq p$. Last, we
generate the data vector $Y$ by $Y \sim N(\theta, I_p)$.

CASE, when applied to the change-point model, requires tuning
parameters $(m, \delta, \tq, \ell^{ps}, \ell^{pe}, u^{pe}, v^{pe})$.
Denote by $s_p \equiv p \eps_p = p^{1-\vartheta}$ the average number of
signals. Given $(s_p, \tau_p)$, we determine the tuning parameters as follows:
As mentioned earlier, we take $m =2$, $\delta= 0$ and $\ell^{ps} = 0$.
Also, we take $\ell^{pe}=10\log(p/s_p)$, $u^{pe}= \sqrt{2\log(p/s_p)}$
and $v^{pe}= \tau_p$.
$\tq$ contains thresholds $t(F,N)$ for each pair of sets $(F,N)$; we
take $t(F,N)=2q(F,N)\log(p)$
with
%
\begin{equation}
\label{q_simu} q(F,N)= 0.8 \times\cases{ %
\bigl(r
\tomega+ |F|\vartheta\bigr)^2/(4r\tomega), &\quad $\tomega>|F|\vartheta,$
\vspace*{2pt}\cr
r\tomega, &\quad $\tomega\leq|F|\vartheta,$}
\end{equation}
where $\vartheta=\log(p/s_p)$, $r=\tau_p^2/(2\log(p))$ and $\tomega
=\tomega(F,N)$ is given in \eqref{tomega(F,N)}.
With these choices, CASE only depends on two parameters $(s_p, \tau_p)$.

\textit{Experiment} 1a. We
compare CASE with the lasso [\citet{Tibshirani}], SCAD [\citet{FanLi}]
(penalty shape parameter $a = 3.7$), MC$+$ [\citet{MCP}] (penalty shape
parameter $\gamma= 1.1$) and SaRa [\citet{SaRa}]. For tuning parameters
$\lambda> 0$ and $h > 0$ (integer), SaRa takes the following form:
\[
\hat{\beta}_i^{\mathrm{SaRa}} = W_i \cdot1
\bigl\{|W_i|> \lambda\bigr\}\qquad \mbox{where } W_i =
\frac{1}{h} \Biggl(\sum_{j=i+1}^{i+h}Y_j
- \sum_{j=i-h+1}^{i}Y_j\Biggr).
\]
The tuning parameters for the lasso, SCAD, MC$+$ and SaRa are ideally set
(pretending we know $\beta$). For CASE,
all tuning parameters depend on $(s_p, \tau_p)$, so we implement the
procedure using the true values of $(s_p, \tau_p)$; this yields
slightly inferior results than that of setting $(s_p, \tau_p)$ ideally
(pretending we know $\beta$, as we do in the lasso, SCAD, MC$+$ and
SaRa), so our comparison in this setting is fair.
Note that even when $(s_p, \tau_p)$ are given, it is unclear how to set
the tuning parameters of the lasso, SCAD, MC$+$ and SaRa.

Fix $p= 5000$ and $a=1$. We let $\vartheta$ range in $\{0.3,0.45,0.6,
0.75\}$ and $\tau_p$ range in $\{3,3.5,\ldots,6.5\}$. The parameters
fall into the regime where exact-recovery is impossible. Table~\ref{tb:simu_changpt1} reports the average Hamming errors of $100$
independent repetitions. We see that CASE consistently outperforms
other methods, especially when $\vartheta$ is small, that is, signals
are less sparse.

We also observe that three \textit{global penalization methods}, the
lasso, SCAD and MC$+$, perform unsatisfactorily, with Hamming errors
comparable to the expected number of signals $s_p$. It suggests that
the \textit{global penalization methods} are not appropriate for the
change-point model when the signals are rare and weak. Similar
conclusions can be drawn in most experiments in this section. To save
space, we only report results of the lasso, SCAD and MC$+$ in this experiment.

\begin{table}
\caption{Comparison of Hamming errors (Experiment \textup{1a}; change-point model)}
\label{tb:simu_changpt1}
\begin{tabular*}{\textwidth}{@{\extracolsep{\fill}}ld{3.1}cd{3.1}d{3.1}d{3.1}d{3.1}d{3.1}d{3.1}@{}}
\hline
\multicolumn{1}{@{}l}{{$\bolds{\vartheta}$}} &
\multicolumn{1}{c}{$\bolds{s_p}$} & &
\multicolumn
{6}{c@{}}{$\bolds{\tau_p}$}\\
\hline
& & & 4.0 & 4.5 & 5.0 & 5.5 & 6.0 & 6.5\\
{0.3} & 338.4 &
CASE & 105.8 & 63.9 & 37.6 & 18.5 & 8.9 & 4.8 \\
& & lasso & 371.7 & 370.0 & 371.5 & 370.1 & 371.5 & 369.8 \\
& & SCAD & 370.6 & 368.3 & 370.5 & 368.2 & 369.3 & 369.2\\
& & MC$+$ & 374.0 & 372.1 & 374.3 & 372.5 & 373.6 & 373.1\\
& & SaRa & 175.6 & 144.0 & 107.8 & 73.7 & 49.0 & 32.3\\[3pt]
& & & 3.0 & 3.5 & 4.0 & 4.5 & 5.0 & 5.5\\
{0.45} &108.3 &
CASE & 50.1 & 35.5 & 26.3 & 20.0 & 12.8 & 6.2 \\
& & lasso & 103.2 & 104.1 & 103.8 & 103.8 & 104.9 & 104.3\\
& & SCAD & 101.8 & 102.7 & 102.1 & 102.0 & 102.9 & 102.5 \\
& & MC$+$ & 103.7 & 104.7 & 104.4 & 104.3 & 105.4 & 104.8 \\
& & SaRa & 78.9 & 72.0 & 66.2 & 63.4 & 61.9 & 60.4 \\[3pt]
& & & 3.0 & 3.5 & 4.0 & 4.5 & 5.0 & 5.5\\
0.6 & 30.2 &
CASE & 14.4 & 11.1 & 8.9 & 6.7 & 5.0 & 3.9 \\
& & lasso & 29.3 & 29.2 & 29.3 & 29.7 & 27.7 & 29.3\\
& & SCAD & 27.7 & 27.7 & 27.9 & 27.4 & 26.1 & 27.1\\
& & MC$+$ & 29.8 & 29.8 & 29.8 & 30.2 & 28.4 & 29.8\\
& & SaRa & 20.4 & 17.0 & 13.6 & 10.9 & 8.6 & 6.8\\[3pt]
& & & 3.0 & 3.5 & 4.0 & 4.5 & 5.0 & 5.5\\
0.75 & 8.4 &
CASE & 3.5 & 2.9 & 2.4 & 1.8 & 1.6 & 1.3\\
& & lasso & 8.2 & 8.3 & 8.5 & 8.8 & 8.0 & 8.5\\
& & SCAD & 6.8 & 7.0 & 7.0 & 6.9 & 6.6 & 6.6\\
& & MC$+$ & 8.7 & 8.8 & 9.1 & 9.2 & 8.7 & 9.1\\
& & SaRa & 5.2 & 4.5 & 3.8 & 3.0 & 2.4 & 2.0\\
\hline
\end{tabular*}
\end{table}

\textit{Experiment} 1b. In this experiment, we investigate the
performance\break of CASE with $(s_p, \tau_p)$ estimated by SaRa; we call
this the \textit{adaptive\break CASE}. In detail, we estimate $(s_p, \tau
_p)$ by
$\hat{s}_p = \sum_{j=1}^p 1\{\hat{\beta}^{\mathrm{SaRa}}_{j}\neq0\}$ and
$\hat
{\tau}_p =\break \operatorname{median} (  \{|\hat{\beta
}^{\mathrm{SaRa}}_j|: \hat
{\beta}_j^{\mathrm{SaRa}}\neq0  \}  )$,
where the tuning parameters $(\lambda, h)$ of SaRa are determined by minimizing
$\operatorname{BIC}(\hat{\beta}) = \frac{1}{2}\Vert Y-X\hat{\beta
}\Vert
^2 + \log(p) \cdot\Vert\hat{\beta}\Vert_0$;
this is a slight modification of Bayesian information criteria (BIC).

For experiment, we use the same setting as in Experiment 1a. Table~\ref{tb:simu_changpt2} reports the average Hamming errors of CASE, SaRa and
the adaptive CASE based on $100$ independent repetitions. First, the
adaptive CASE [CASE but $(s_p, \tau_p)$ are estimated by SaRa] has a
very similar performance to CASE. Second, although the adaptive CASE
uses SaRa as the preliminary estimator, its performance is
substantially better than that of SaRa (and other methods in the same
setting; see Experiment~1a).

\begin{table}
\caption{Comparison of Hamming errors (Experiment \textup{1b}). ``adCASE'' stands
for adaptive CASE, where $(s_p, \tau_p)$ are estimated from SaRa [where
$(\lambda, h)$ are set by BIC]} \label{tb:simu_changpt2}
\begin{tabular*}{\textwidth}{@{\extracolsep{\fill}}ld{3.1}cd{3.1}d{3.1}d{3.1}d{3.1}d{3.1}d{3.1}@{}}
\hline
\multicolumn{1}{@{}l}{{$\bolds{\vartheta}$}} &
\multicolumn{1}{c}{$\bolds{s_p}$} & &
\multicolumn
{6}{c@{}}{$\bolds{\tau_p}$}\\
\hline
& & & 4.0 & 4.5 & 5.0 & 5.5 & 6.0 & 6.5\\
0.3 & 338.4 &
CASE & 105.8 & 63.9 & 37.6 & 18.5 & 8.9 & 4.8 \\
& & adCASE & 100.3 & 63.6 & 37.8 & 18.6 & 8.9 & 4.8 \\
& & SaRa & 190.7 & 162.0 & 131.3 & 98.0 & 68.2 & 47.1\\[3pt]
& & & 3.0 & 3.5 & 4.0 & 4.5 & 5.0 & 5.5\\
0.45 & 108.3 &
CASE & 50.1 & 35.5 & 26.3 & 20.0 & 12.8 & 6.2 \\
& & adCASE & 48.6 & 33.9 & 26.0 & 20.8 & 16.6 & 9.7 \\
& & SaRa & 86.1 & 76.7 & 71.4 & 66.7 & 65.0 & 62.8 \\[3pt]
& & & 3.0 & 3.5 & 4.0 & 4.5 & 5.0 & 5.5\\
0.6 & 30.2 &
CASE & 14.4 & 11.1 & 8.9 & 6.7 & 5.0 & 3.9 \\
& & adCASE & 14.0 & 11.0 & 8.8 & 6.5 & 4.8 & 3.4\\
& & SaRa & 35.7 & 28.5 & 24.1 & 19.9 & 15.8 & 11.9\\[3pt]
& & & 3.0 & 3.5 & 4.0 & 4.5 & 5.0 & 5.5\\
0.75 & 8.4 &
CASE & 3.5 & 2.9 & 2.4 & 1.8 & 1.6 & 1.3\\
& & adCASE & 3.7 & 3.0 & 2.2 & 1.8 & 1.5 & 1.3\\
& & SaRa & 13.3 & 11.5 & 8.0 & 5.2 & 4.0 & 2.9\\
\hline
\end{tabular*}
\end{table}

\textit{Experiment} 2. In this experiment, we consider the
post-filtering model, model~(\ref{modeladd}), associated with the
change-point model, and illustrate that the seeming simplicity of this
model (where $DG = I_p$, and $DGD'$ is tri-diagonal) does not mean
it is a trivial setting for variable selection. In particular, if we
naively apply the $L^0/L^1$-penalization to the post-filtering model,
we end up with naive soft/hard thresholding; we illustrate our point by
showing that CASE significantly outperforms
naive thresholding (since we use Hamming distance as the loss function,
there is no difference between soft and hard thresholding). For both
CASE and
naive thresholding, we set tuning parameters assuming $(s_p, \tau_p)$
as known. The threshold of naive thresholding is set as $(r+2\vartheta
)^2/(2r)\cdot\log(p)$, where $\vartheta=\log(p/s_p)$ and $r=\tau
_p^2/(2\log(p))$; this threshold choice is known as theoretically optimal.

Fix $p=10^6$ and $a=1$ (so that the signals have equal strengths). Let
$\vartheta$ range in $\{0.35, 0.5, 0.75\}$, and $\tau_p$ range in $\{
5,\ldots, 13\}$.
Table~\ref{tb:simu_changpt3} reports the average Hamming errors of $50$
independent repetitions, which show that CASE outperforms naive
thresholding in most cases, especially when $\vartheta$ is small or
$\tau_p$ is small. It
suggests that the post-filtering model remains largely nontrivial, and
to deal with it, we need sophisticated methods.

\begin{table}
\tabcolsep=0pt
\caption{Comparison of Hamming errors (Experiment 2; change-point
model), $p=10^6$. ``nHT'' stands for naive hard thresholding} \label
{tb:simu_changpt3}
\begin{tabular*}{\textwidth}{@{\extracolsep{\fill}}ld{4.0}cd{4.1}d{4.1}d{4.1}d{3.1}
d{3.1}d{3.1}d{2.1}d{2.1}c@{}}
\hline
 &
 & &
\multicolumn
{9}{c@{}}{$\bolds{\tau_p}$}\\[-6pt]
&&&
\multicolumn{9}{c@{}}{\hrulefill}\\
\multicolumn{1}{@{}l}{$\bolds{\vartheta}$}&\multicolumn{1}{c}{{$\bolds{s_p}$}} & & \textbf{5} & \textbf{6} & \textbf{7} &
\textbf{8} & \textbf{9} & \textbf{10} & \textbf{11} & \textbf{12} &
\multicolumn{1}{c@{}}{\textbf{13}}\\
\hline
0.35 & 7943 &
CASE & 956.7 & 332.6 & 117.5 & 49.1 & 24.1 & 13.9 & 10.6 & 7.7 & 7.3 \\
& & nHT & 4430.5 & 2381.3 & 1085.8 & 418.1 & 139.7 & 41.9 & 11.0 & 2.5
& 0.5 \\[3pt]
0.50 & 1000 &
CASE & 195.3 & 68.8 & 20.8 & 5.0 & 1.3 & 0.7 & 0.4 & 0.1 & 0.2\\
& & nHT & 767.9 & 489.0 & 250.8 & 105.3 & 38.4 & 12.4 & 3.5 & 0.7 &
0.2\\[3pt]
0.75 & 32 & CASE & 9.3 & 3.1 & 2.3 &
0.4 & 0.1 & 0.1 & 0.1 & 0.0 & 0.0\\
& & nHT & 31.1 & 25.6 & 15.7 & 8.3 & 3.2 & 1.8 & 0.5 & 0.0 & 0.0\\
\hline
\end{tabular*}
\end{table}

\textit{Experiment} 3. In this experiment, we fix $(p,\vartheta,
\tau
_p)= (5000, 0.50, 4.5)$, and let $a$ range in $\{1,1.5,\ldots, 3\}$ (so
signals may have different strengths). We investigate a case where the
signals have the ``half-positive-half-negative'' sign pattern, that is,
$\beta_j\stackrel{\mathrm{i.i.d.}}{\sim}(1-\eps_p)\nu_0 + \frac{\eps_p}{2}
U(\tau
_p, a\tau_p) + \frac{\eps_p}{2} U(-a\tau_p, - \tau_p)$, and a case
where the signals have the ``all-positive'' sign pattern, that is,
$\beta_j\stackrel{\mathrm{i.i.d.}}{\sim}(1-\eps_p)\nu_0 + \eps_p U(\tau_p,
a\tau_p)$.
We compare CASE with SaRa for different values of $a$ and sign patterns
(we do not include the lasso, SCAD, MC$+$ in this particular experiment,
for at least for the experiments reported above, they are inferior to SaRa).
The tuning parameters for both CASE and SaRa are set ideally as in
Experiment 1a.
The results of $50$ independent repetitions are reported in Table~\ref{tb:simu1_exp4}, which suggest that CASE uniformly outperforms SaRa for
various values of $a$ and the two sign patterns.

\begin{table}[b]
\caption{Comparison of Hamming errors (Experiment 3; change-point
model) for different choices of $a$ (the ratio between the maximum and
minimum signal strengths) and for two sign patterns ``half--half'' and
``all positive.'' $p=5000$, $\vartheta= 0.5$, $s_p =70.7$ and $\tau
_p=4.5$}
 \label{tb:simu1_exp4}
\begin{tabular*}{\textwidth}{@{\extracolsep{\fill}}lcd{2.2}d{2.2}d{2.2}d{2.2}d{2.2}@{}}
\hline
& & \multicolumn{5}{c@{}}{$\bolds{a}$} \\ [-6pt]
& & \multicolumn{5}{c@{}}{\hrulefill} \\
& & \multicolumn{1}{c}{\textbf{1}} & \multicolumn{1}{c}{\textbf{1.5}} & \multicolumn{1}{c}{\textbf{2}} & \multicolumn{1}{c}{\textbf{2.5}} & \multicolumn{1}{c@{}}{\textbf{3}} \\
\hline
{Half--half} & CASE & 14.26 & 6.32 & 5.50 & 4.78 & 4.56
\\
& SaRa & 24.98 & 18.96 & 16.56 & 14.00 & 12.50 \\[3pt]
{All-positive} & CASE & 13.44 & 6.18 & 4.90 & 5.38 &
4.14 \\
& SaRa & 24.26 & 18.58 & 16.80 & 13.66 & 12.12\\
\hline
\end{tabular*}
\end{table}

\subsection{Long-memory time series model} \label{subsec:simutoep}
In this section, we investigate long-memory time series, focusing on
the $\operatorname{FARIMA}(0,\phi,0)$ process [\citet{FanTimeS}], where
$\phi$ is the long-memory parameter.
We let $X = G^{1/2}$ where $G$ is constructed according to \eqref
{lmG}--\eqref{f in FARIMA}. For $\beta$ generated in ways to be
specified, we let $Y \sim N(X\beta, I_p)$.

CASE uses tuning parameters $(m,\delta, \tq, \ell^{ps}, \ell^{pe},
u^{pe}, v^{pe})$, which are set in the same way as in the change-point
model, except for two differences. First, different from that in the
change-point model,
we need a regularization parameter $\delta$ which we set as $2.5/\log
(p)$. Second, we take $\ell^{ps} = \ell^{pe}/2$.

\textit{Experiment} 4a. In this experiment, we compare CASE with the
lasso, SCAD (shape parameter $a=3.7$) and MC$+$ (shape parameter $\gamma=2$).
Fixing $p=5000$ and $\phi=0.35$, we let $\vartheta$ range in $\{ 0.35,
0.45, 0.55\}$, and let $\tau_p$ range in
$\{ 4,\ldots, 8 \}$. For each pair of $(\vartheta, \tau_p)$, we
generate the vector $\beta$ by $\beta_j \stackrel{\mathrm{i.i.d.}}{\sim}
(1-\eps
_p)\nu_0 + \frac{\eps_p}{2} \nu_{\tau_p} + \frac{\eps_p}{2} \nu
_{-\tau_p}$.
Similar to that in Experiment 1a, the tuning parameters of CASE are set
assuming $(s_p, \tau_p)$ as known, and the tuning parameters of the
lasso, SCAD and MC$+$ are set ideally to minimize the Hamming error
(assuming $\beta$ is known). By similar argument as in Experiment 1a,
the comparison is fair.
Table~\ref{tb:simu_toep1} reports the average Hamming errors based on
$100$ independent repetitions. The results suggest that CASE
outperforms the lasso and SCAD, and has a comparable performance to
that of MC$+$.

\begin{table}
\caption{Comparison of Hamming errors (Experiment \textup{4a}). The Gram matrix
is the population covariance matrix of the $\operatorname
{FARIMA}(0,\phi
,0)$ process with $\phi=0.35$. $p=5000$}
 \label{tb:simu_toep1}
\begin{tabular*}{\textwidth}{@{\extracolsep{\fill}}ld{3.1}cd{3.1}d{2.1}d{2.1}d{2.1}d{2.1}@{}}
\hline
 &  & & \multicolumn{5}{c@{}}{$\bolds{\tau_p}$}\\[-6pt]
  &  & & \multicolumn{5}{c@{}}{\hrulefill}\\
\multicolumn{1}{@{}l}{{$\bolds{\vartheta}$}}& \multicolumn{1}{c}{{$\bolds{s_p}$}}& &
 \multicolumn{1}{c}{\textbf{4}} & \multicolumn{1}{c}{\textbf{5}} & \multicolumn{1}{c}{\textbf{6}} & \multicolumn{1}{c}{\textbf{7}} & \multicolumn{1}{c@{}}{\textbf{8}} \\
\hline
{0.35} &253.7
& CASE & 118.0 & 60.7 & 26.3 & 9.5 & 4.3\\
& & lasso & 145.2 & 91.6 & 60.2 & 37.4 & 26.0\\
& & SCAD & 140.6 & 87.0 & 42.8 & 19.5 & 8.0\\
& & MC$+$ & 108.6 & 50.2 & 20.4 & 7.4 & 2.6\\[3pt]
0.45 & 108.3
& CASE & 60.3 & 27.7 & 11.8 & 4.0 & 1.9\\
& & lasso & 65.6 & 40.0 & 23.2 & 13.5 & 7.7\\
& & SCAD & 64.0 & 37.7 & 19.6 & 9.2 & 3.9\\
& & MC$+$ & 52.0 & 23.6 & 8.6 & 3.0 & 1.0\\[3pt]
0.55 & 46.2
& CASE & 27.9 & 13.4 & 4.3 & 1.4 & 0.5\\
& & lasso & 27.8 & 16.0 & 8.0 & 3.9 & 2.1\\
& & SCAD & 27.0 & 15.2 & 7.0 & 3.1 & 1.2\\
& & MC$+$ & 23.4 & 10.6 & 3.1 & 0.7 & 0.2\\
\hline
\end{tabular*}
\end{table}

\textit{Experiment} 4b. We investigate the setting where ``signal
cancellation'' is more severe
than that in Experiment 4b. Toward this end,
we use the same setting as in Experiment 4a, except for
that $\beta$ is generated in a way that signals appear in adjacent
pairs with opposite signs,
$(\beta_{2j-1}, \beta_{2j})\stackrel{\mathrm{i.i.d.}}{\sim} (1-\eps_p)\nu
_{(0,0)} +
\eps_p\nu_{(\tau_p, -\tau_p)}$, $1 \leq j \leq p/2$, where $\nu
_{(a,b)}$ is the point mass at $(a,b)\in\mathbb{R}^2$. Hamming errors
based on $100$ repetitions are reported in Table~\ref{tb:simu_toep2},
suggesting that CASE significantly outperforms all the other methods.

It is noteworthy that MC$+$ behaves much more unsatisfactorily here than
in Experiment 4a, and the main reason is that MC$+$ does not adequately
address ``signal cancellation.'' In contrast, one of the major
advantages of CASE is that it addresses adequately the ``signal
cancellation''; this is why it has satisfactory performance in both
Experiments 4a and 4b.

%
\begin{table}
\caption{Comparison of Hamming errors (Experiment \textup{4b})}
 \label{tb:simu_toep2}
\begin{tabular*}{\textwidth}{@{\extracolsep{\fill}}ld{3.1}cd{3.1}d{2.1}d{2.1}d{2.1}d{2.1}@{}}
\hline
 &  & & \multicolumn{5}{c@{}}{$\bolds{\tau_p}$}\\[-6pt]
  &  & & \multicolumn{5}{c@{}}{\hrulefill}\\
\multicolumn{1}{@{}l}{{$\bolds{\vartheta}$}}& \multicolumn{1}{c}{{$\bolds{s_p}$}}& &
 \multicolumn{1}{c}{\textbf{4}} & \multicolumn{1}{c}{\textbf{5}} & \multicolumn{1}{c}{\textbf{6}} & \multicolumn{1}{c}{\textbf{7}} & \multicolumn{1}{c@{}}{\textbf{8}} \\
\hline
0.35 & 253.7
& CASE & 138.6 & 60.8 & 23.3 & 7.2 & 1.8\\
& & lasso & 223.0 & 158.9 & 97.9 & 54.8 & 27.1\\
& & SCAD & 257.5 & 156.8 & 95.1 & 52.1 & 25.1\\
& & MC$+$ & 206.7 & 129.2 & 68.6 & 33.4 & 13.6\\[3pt]
0.45 & 108.3
& CASE & 75.7 & 36.4 & 13.3 & 3.7 & 0.9\\
& & lasso & 100.0 & 84.7 & 58.4 & 32.2 & 15.9\\
& & SCAD & 99.2 & 83.2 & 56.6 & 30.6 & 14.9\\
& & MC$+$ & 98.1 & 76.0 & 44.8 & 21.5 & 8.9\\[3pt]
0.55 & 46.2
& CASE & 38.6 & 20.0 & 8.9 & 3.6 & 1.0\\
& & lasso & 45.4 & 40.1 & 31.0 & 20.6 & 10.9\\
& & SCAD & 45.0 & 39.4 & 30.1 & 19.6 & 9.9\\
& & MC$+$ & 44.9 & 38.4 & 26.3 & 14.8 & 6.8\\
\hline
\end{tabular*}
\end{table}

\begin{table}[b]
\caption{Hamming errors for CASE, applied when $(s_p, \tau_p)$ are
misspecified as $p^{1 - \tilde{\vartheta}}$ and $\tilde{\tau}_p$,
respectively (Experiment 5). The Gram matrix is the population
covariance matrix of the $\operatorname{FARIMA}(0,\phi,0)$ process with
$\phi=0.35$. $p=5000$}
 \label{tb:simu_toepmisspec}
\begin{tabular*}{\textwidth}{@{\extracolsep{\fill}}lcd{2.1}d{2.2}d{2.1}d{2.2}d{2.1}d{2.2}d{2.1}@{}}
\hline
$\vartheta=0.35, \tau_p=6$ & $\tilde{\vartheta}$ & 0.2 & 0.25 &
0.3 & 0.35 & 0.4 & 0.45 & 0.5 \\
$s_p=253.7$& & 27.8 & 24.8 & 23.2 & 23.2 & 24.5 & 26.3 & 48.9\\[3pt]
& $\tilde{\tau}_p$ & 4 & 5 & 5.5 & 6 & 6.5 & 7 & 8\\
& & 47.3 & 30.2 & 25.3 & 23.2 & 23.9 & 26.9 & 42.7\\[6pt]
$\vartheta=0.55, \tau_p=5$
& $\tilde{\vartheta}$ & 0.4 & 0.45 & 0.5 & 0.55 & 0.6 & 0.65 & 0.7 \\
$s_p=46.2$& & 21.8 & 19.0 & 19.3 & 19.8 & 21.7 & 25.5 & 25.4\\[3pt]
& $\tilde{\tau}_p$ & 3 & 4 & 4.5 & 5 & 5.5 & 6 & 7\\
& & 23.8 & 22.2 & 20.8 & 19.8 & 21.0 & 23.9 & 29.0\\
\hline
\end{tabular*}
\end{table}

\textit{Experiment} 5. In some of the experiments above, we set the
tuning parameters of CASE assuming
$(s_p, \tau_p)$ as known. It is therefore interesting to investigate
how the misspecification of $(s_p, \tau_p)$ affects the performance of
CASE. Fix $p=5000$ and $\phi=0.35$. We consider two combinations of
$(\vartheta, \tau_p)$: $(\vartheta, \tau_p) = (0.35, 6), (0.55, 5)$.
The vector $\beta$ is generated in the same way as in Experiment 4b,
with the signals appearing in adjacent pairs.
We fix one parameter of $(s_p, \tau_p)$ and mis-specify the other
[since $s_p$ is not on the same scale as $\tau_p$, the results are
reported based on the misspecification of $(\vartheta, \tau_p)$,
instead of $(s_p, \tau_p)$; recall here $s_p = p^{1 - \vartheta}$]. We
then apply CASE with tuning parameters set based on the misspecified
values of $(s_p, \tau_p)$. Table~\ref{tb:simu_toepmisspec} reports the
average Hamming errors of $50$ independent repetitions, which is a
rather flat function of misspecified values of $\vartheta$ (with $\tau
_p$ fixed) or
of misspecified values of $\tau_p$ (with $\vartheta$ fixed).
In comparison, the Hamming errors of the lasso are 97.9 and 40.1 in the
two settings, respectively, with the tuning parameter ideally set as in
Experiment 1a.
This suggests that CASE is relatively insensitive to the
misspecification of $(s_p, \tau_p)$, and outperforms the lasso as long
as the misspecification of $(s_p, \tau_p)$ is not severe.

\begin{figure}[b]

\includegraphics{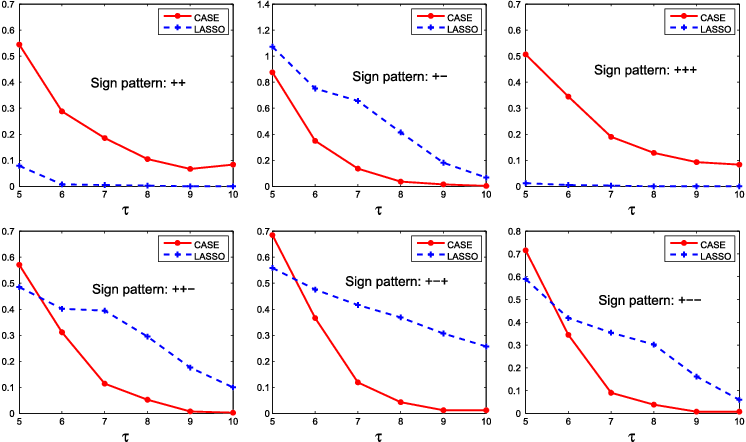}

\caption{Comparison of Hamming errors (Experiment 6).
The Gram matrix is the population covariance matrix of the
$\operatorname{FARIMA}(0,\phi,0)$
process with $\phi=0.35$. $p=5000$, $\vartheta=0.75$ and $s_p=32$.
Signals appear in adjacent pairs (Panels 1--2) or adjacent triplets
(Panels 3--6),
with different sign patterns specified in the panels, associated with
different level of
``signal cancellation.'' CASE outperforms lasso when signal cancellation
is severe.} \label{fig:simu2_exp3}
\end{figure}

\textit{Experiment} 6. In this experiment, we continue to investigate
the effect of
``signal cancellation,'' where we compare CASE with the lasso using
several choices of signal patterns (and so
the levels of ``signal cancellation'' vary).
Fix $p= 4998$, $\phi=0.35$, $\vartheta= 0.75$, and let $\tau_p$ range
in $\{ 5,\ldots, 10\}$.
The experiment contains two parts. In the first part, we partition $\{
1, 2,\ldots, p\}$ into $p/2$ blocks, where each block contains two
adjacent indices. In $(1 - \eps_p)$ fraction of the blocks, $\beta_j =
0$. In the $\eps_p$ fraction of the blocks, we have either that $\beta
_j = \tau_p$ for both $j$ in the block (we denote the sign pattern by
``$++$''), or that $\beta_j = \tau_p$ if $j$ is the first index in the
block and $\beta_j = - \tau_p$ otherwise (sign pattern ``$+-$''). In
the second part, we partition $\{1, 2,\ldots, p\}$ into $p/3$ blocks
(block size is $3$), and
generate $\beta$ in a similar fashion, but with four different sign
patterns in each block: ``$+++$,'' ``$++-$,'' ``$+-+$'' and ``$+--$.''
Hamming errors based on $50$ independent repetitions are reported in
Figure~\ref{fig:simu2_exp3}. The results suggest that (a) when the sign
patterns are ``$++$'' or ``$+++$,''
signal cancellation has negligible effects, and CASE is inferior to the
lasso, (b) when the sign patterns are
``$+-$,'' ``$++-$,'' ``$+-+$'' and ``$+--$,'' the effect of signal
cancellation kicks in, and CASE outperforms the lasso in most cases.
This is consistent with our theoretic insight that CASE adequately
addresses signal cancellation, but the lasso does not.

\section{Discussion} \label{sec:Discu}
Variable selection when the Gram matrix $G$ is nonsparse
is a challenging problem. We approach this problem by first
sparsifying $G$ with a finite order linear filter, and then
constructing a sparse graph GOSD. The key insight is that,
in the post-filtering data, the true signals live in many
small-size components that are disconnected in GOSD, but we do not know where.
We propose CASE as a new approach to variable selection.
This is a two-stage screen and clean method, where we first use a
covariance-assisted multivariate screening to identify candidates for
such small-size components, and then re-examine each candidate with
penalized least squares. In both stages, to overcome the problem of
information leakage, we employ a delicate patching technique.

We develop an asymptotic framework focusing on the regime where the signals
are rare and weak so that successful variable selection is challenging
but is
still possible. We show that CASE achieves the optimal rate of
convergence in Hamming distance across a wide class of situations where
$G$ is nonsparse but sparsifiable.
Such optimality cannot be achieved by many popular methods, including
but not
limited to the lasso, SCAD, MC$+$ and Dantzig selector. When $G$ is
nonsparse, these methods are not expected to behave well even when the
signals are strong.
We have successfully applied CASE to two different applications: the
change-point problem and the long-memory times series.

Compared to the well-known method of marginal screening [\citet{FanLv,Wasserman}],
CASE employs a covariance-assisted multivariate screening procedure, so
that it is theoretically more effective than marginal screening, with
only a moderate increase in the computational complexity.
CASE is closely related to the graphical lasso [\citet{graphiclasso,Meinshausen}],
which also attempts to exploit the graph structure. However,
the setting considered here is very different from that in \citet
{graphiclasso} and \citet{Meinshausen},
and our emphasis on optimality is also very different.

The paper is closely related to the recent work \citet{JZ11} [see also
\citet{UPS}], but is different in important ways. The work in \citet
{JZ11} is motivated by recent literature of compressive sensing and
genetic regulatory network, and is largely focused on the
case where the Gram matrix $G$ is sparse in an unstructured fashion.
The current work is motivated
by the recent interest on DNA-copy number variation and long-memory time
series, and is focused on the case where there are strong dependence
between different design variables, so $G$ is usually nonsparse and sometimes
ill-posed. To deal with the
strong dependence, we have to use a finite-order linear filter and
delicate patching techniques. Additionally, the current paper also
studies applications to the long-memory
time series and the change-point problem which have not been considered
in \citet{JZ11}.
Especially, the studies on the change-point problem encompasses very
different and
very delicate analysis on both the derivation of the lower bound and
upper bound which we have not seen before in the literature.
For these reasons, the two papers have very different scopes and
techniques, and
the results in one paper cannot be deduced from those in the other.

In this paper, we are primarily interested in the linear model, model
\eqref{model}, but
CASE is applicable in much broader settings.
For example, in model \eqref{model},
we assume that the coordinates of $z$ have the same variance $\sigma
^2$, and $\sigma$ is known (and so without loss of generality, we
assume $\sigma= 1$). When $\sigma$ is unknown, the main results in
this paper continue to hold, provided that we can estimate $\sigma$
consistently [say, except for a probability of $o(p^{-2})$, there is an
estimate $\hat{\sigma}$ such that $|\hat{\sigma}/\sigma- 1| = o(1)$].
Such an estimator can be obtained by adapting the scaled-lasso approach
by \citet{SunZhang} or the refitted cross validation by \citet{RCV} to
the post-filtering model \eqref{modeladd}. Correspondingly, we need to
modify the tuning parameters of CASE slightly. For example, in the
$\mathit{PS}$-step, $\mathcal{Q}$ is replaced by $\hat{\sigma}^2\mathcal
{Q}\equiv\{
\hat{\sigma}^2 t(F,N)\}$, and in the
$\mathit{PE}$-step, $u^{pe}$ and $v^{pe}$ are replaced by $\hat{\sigma}u^{pe}$ and
$\hat{\sigma} v^{pe}$, respectively.

Also, in model \eqref{model}, we have assumed that the coordinates of
$z$ are Gaussian distributed.
Such an assumption can also be relaxed. In fact, in the core of CASE is
the analysis of low-dimensional sub-vectors of $\tilde{Y}=X'Y$, where
we note that each coordinate of $\tilde{Y}$ has the form of
$b_0 + a' z$
for some constant $b_0$ and $n \times1$ nonstochastic vector $a$.
Note that $a$ only depends on
the design matrix and the index of the coordinate of $\tilde{Y}$ (so
there are $p$ different vectors $a$ at most).
Essentially, the Gaussian assumption is only
required for $a' z$ for all $p$ different choices of $a$. Note that
even when $z$ is non-Gaussian,
$a'z$ could be approximately Gaussian by central limit theorem; this
holds, for example, for the long-memory time series considered in the
paper. As a result, the Gaussian assumption on $z$ can be largely relaxed.

The main results in this paper can be extended in many other directions.
For example, we have used a rare and weak signal model
where the signals are randomly generated from a two-component
mixture. The main results continue to hold if we choose to use
a much more general model,
as long as the signals live in small-size
isolated ``islands'' in the post-filtering data, where each island is a
connected subgraph of GOSD.

Also, we have been focused on the change-point model and the long-memory
time series model, where the post-filtering matrices have polynomial
off-diagonal
decay and are sparse in a structured fashion.
CASE can be extended to more general settings, where the sparsity of
the post-filtering matrices are unstructured, provided that we modify
the patching technique accordingly: the patching set can be constructed
by including nodes which are connected to the original set through a
short-length path in the GOSD.

Still another extension is that the Gram matrix can be sparsified by an
operator~$D$, but $D$ is not necessary linear filtering. To apply CASE
to this setting, we need to design specific patching technique. For
example, when $D^{-1}$ is sparse, for a given $\call$, we can construct
$\call^{pe}=\{j\dvtx |D^{-1}(i,j)|>\delta_1, \operatorname{ for some }
i\in
\call\}$, where $\delta_1$ is a chosen threshold.

The paper is closely related to recent literature on DNA copy number
variation and financial data analysis, and it
is of interest to further investigate such connections. To save space,
we leave
explorations along this line to the future.

A key component of our approach is the notion of ``sparsifiability,''
meaning that we can make the Gram matrix $G$ sparse by some simple operations.
Usually, to find such a simple operation, we need a good understanding
about the structure of $G$. In some applications, it is not hard to
find such an operation:
\begin{itemize}
\item In compressive sensing or genome-wide association study (GWAS),
where the rows of the design matrix $X$ are i.i.d. samples from a
$p$-dimensional distribution of zero means and a sparse covariance
matrix $\Sigma$. In Compressive Sensing, $\Sigma$ is usually
proportional to the identity matrix, and in GWAS, $\Sigma$ is a banded
matrix. In such examples, $G$ is already sparse, and so sparsifiable.
\item The current paper is largely motivated by the change-point
problem and long-memory time series models, where $G$ can be
sparsifiable by a liner filter $D$.
\item Another example of sparsifiability is that $G$ is the sum of a
symmetric low-rank matrix $L$ and a sparse matrix $S$; when spectral
norm of $S$ is much smaller than the smallest nonzero eigenvalue of
$L$, $G$ is sparsifiable
by principle component analysis (PCA).
\end{itemize}
In more complicated case where we have little understanding about the
structure~$G$, how to
sparsify it with a simple operation is a nontrivial problem, though we
can always try linear filtering, PCA or both.
We leave the study along this line for future.

The notion of ``sparsifiability'' may apply to both nonrandom design
models and random design models.
For a random-design model where rows of $X$ are i.i.d. samples from a
$p$-dimensional distribution with zero means and covariance matrix~$\Sigma$,
that $G$ is sparsifiable usually means that $\Sigma$ is
sparsifiable.
In more general case, how to sparsify the Gram matrix $G$ is a
nontrivial problem, and we leave such discussions to the future work.

\section*{Acknowledgments}
The authors would like to thank Ning Hao, Philippe Rambour and David
Siegmund for helpful pointers and comments.


\begin{supplement}[id=suppA]
\stitle{\spaceskip=0.6em plus 0.06em minus 0.06em{Supplement to ``Covariate assisted screening and estimation''}}
\slink[doi]{10.1214/14-AOS1243SUPP} 
\sdatatype{.pdf}
\sfilename{aos1243\_supp.pdf}
\sdescription{Owing to space constraints, the technical proofs are
relegated a
supplementary document. It contains Sections A--C.}
\end{supplement}

%

\printaddresses

\begin{thebibliography}{36}


\bibitem[\protect\citeauthoryear{Andreou and Ghysels}{2002}]{Andr02}
\begin{barticle}[author]
\bauthor{\bsnm{Andreou},~\bfnm{Elena}\binits{E.}} \AND
\bauthor{\bsnm{Ghysels},~\bfnm{Eric}\binits{E.}}
(\byear{2002}).
\btitle{Detecting multiple breaks in financial market volatility dynamics}.
\bjournal{J. Appl. Econometrics}
\bvolume{17}
\bpages{579--600}.
\end{barticle}
\bptok{imsref}%
\endbibitem

\bibitem[\protect\citeauthoryear{Bhattacharya}{1994}]{ChangptReview}
\begin{bincollection}[mr]
\bauthor{\bsnm{Bhattacharya},~\bfnm{P.~K.}\binits{P.~K.}}
(\byear{1994}).
\btitle{Some aspects of change-point analysis}.
In \bbooktitle{Change-Point Problems ({S}outh {H}adley, MA, 1992)}.
\bseries{Institute of Mathematical Statistics Lecture Notes---Monograph Series}
\bvolume{23}
\bpages{28--56}.
\bpublisher{IMS},
\blocation{Hayward, CA}.
\bid{doi={10.1214/lnms/1215463112}, mr={1477912}}
\end{bincollection}
\bptok{imsref}%
\endbibitem

\bibitem[\protect\citeauthoryear{Cand{\`e}s and Plan}{2009}]{Candes}
\begin{barticle}[mr]
\bauthor{\bsnm{Cand{\`e}s},~\bfnm{Emmanuel~J.}\binits{E.~J.}} \AND
\bauthor{\bsnm{Plan},~\bfnm{Yaniv}\binits{Y.}}
(\byear{2009}).
\btitle{Near-ideal model selection by {$\ell_1$} minimization}.
\bjournal{Ann. Statist.}
\bvolume{37}
\bpages{2145--2177}.
\bid{doi={10.1214/08-AOS653}, issn={0090-5364}, mr={2543688}}
\end{barticle}
\bptok{imsref}%
\endbibitem

\bibitem[\protect\citeauthoryear{Chen, Hurvich and Lu}{2006}]{ChenHur}
\begin{barticle}[mr]
\bauthor{\bsnm{Chen},~\bfnm{Willa~W.}\binits{W.~W.}},
\bauthor{\bsnm{Hurvich},~\bfnm{Clifford~M.}\binits{C.~M.}} \AND
\bauthor{\bsnm{Lu},~\bfnm{Yi}\binits{Y.}}
(\byear{2006}).
\btitle{On the correlation matrix of the discrete {F}ourier transform and the fast solution of large {T}oeplitz systems for long-memory time series}.
\bjournal{J. Amer. Statist. Assoc.}
\bvolume{101}
\bpages{812--822}.
\bid{doi={10.1198/016214505000001069}, issn={0162-1459}, mr={2281249}}
\end{barticle}
\bptok{imsref}%
\endbibitem

\bibitem[\protect\citeauthoryear{Donoho and Huo}{2001}]{DonohoHuo}
\begin{barticle}[mr]
\bauthor{\bsnm{Donoho},~\bfnm{David~L.}\binits{D.~L.}} \AND
\bauthor{\bsnm{Huo},~\bfnm{Xiaoming}\binits{X.}}
(\byear{2001}).
\btitle{Uncertainty principles and ideal atomic decomposition}.
\bjournal{IEEE Trans. Inform. Theory}
\bvolume{47}
\bpages{2845--2862}.
\bid{doi={10.1109/18.959265}, issn={0018-9448}, mr={1872845}}
\end{barticle}
\bptok{imsref}%
\endbibitem

\bibitem[\protect\citeauthoryear{Donoho and Jin}{2008}]{DonohoJin2008}
\begin{barticle}[author]
\bauthor{\bsnm{Donoho},~\bfnm{David}\binits{D.}} \AND
\bauthor{\bsnm{Jin},~\bfnm{Jiashun}\binits{J.}}
(\byear{2008}).
\btitle{Higher criticism thresholding: Optimal feature selection when useful features are rare and weak}.
\bjournal{Proc. Natl. Acad. Sci. USA}
\bvolume{105}
\bpages{14790--14795}.
\end{barticle}
\bptok{imsref}%
\endbibitem

\bibitem[\protect\citeauthoryear{Donoho and Stark}{1989}]{DonohoStarck}
\begin{barticle}[mr]
\bauthor{\bsnm{Donoho},~\bfnm{David~L.}\binits{D.~L.}} \AND
\bauthor{\bsnm{Stark},~\bfnm{Philip~B.}\binits{P.~B.}}
(\byear{1989}).
\btitle{Uncertainty principles and signal recovery}.
\bjournal{SIAM J. Appl. Math.}
\bvolume{49}
\bpages{906--931}.
\bid{doi={10.1137/0149053}, issn={0036-1399}, mr={0997928}}
\end{barticle}
\bptok{imsref}%
\endbibitem

\bibitem[\protect\citeauthoryear{Fan, Guo and Hao}{2012}]{RCV}
\begin{barticle}[mr]
\bauthor{\bsnm{Fan},~\bfnm{Jianqing}\binits{J.}},
\bauthor{\bsnm{Guo},~\bfnm{Shaojun}\binits{S.}} \AND
\bauthor{\bsnm{Hao},~\bfnm{Ning}\binits{N.}}
(\byear{2012}).
\btitle{Variance estimation using refitted cross-validation in ultrahigh dimensional regression}.
\bjournal{J. R. Stat. Soc. Ser. B Stat. Methodol.}
\bvolume{74}
\bpages{37--65}.
\bid{doi={10.1111/j.1467-9868.2011.01005.x}, issn={1369-7412}, mr={2885839}}
\end{barticle}
\bptok{imsref}%
\endbibitem

\bibitem[\protect\citeauthoryear{Fan and Li}{2001}]{FanLi}
\begin{barticle}[mr]
\bauthor{\bsnm{Fan},~\bfnm{Jianqing}\binits{J.}} \AND
\bauthor{\bsnm{Li},~\bfnm{Runze}\binits{R.}}
(\byear{2001}).
\btitle{Variable selection via nonconcave penalized likelihood and its oracle properties}.
\bjournal{J. Amer. Statist. Assoc.}
\bvolume{96}
\bpages{1348--1360}.
\bid{doi={10.1198/016214501753382273}, issn={0162-1459}, mr={1946581}}
\end{barticle}
\bptok{imsref}%
\endbibitem

\bibitem[\protect\citeauthoryear{Fan and Song}{2010}]{FanLv}
\begin{barticle}[mr]
\bauthor{\bsnm{Fan},~\bfnm{Jianqing}\binits{J.}} \AND
\bauthor{\bsnm{Song},~\bfnm{Rui}\binits{R.}}
(\byear{2010}).
\btitle{Sure independence screening in generalized linear models with NP-dimensionality}.
\bjournal{Ann. Statist.}
\bvolume{38}
\bpages{3567--3604}.
\bid{doi={10.1214/10-AOS798}, issn={0090-5364}, mr={2766861}}
\end{barticle}
\bptok{imsref}%
\endbibitem

\bibitem[\protect\citeauthoryear{Fan, Xue and Zou}{2014}]{FXZ}
\begin{barticle}[author]
\bauthor{\bsnm{Fan},~\bfnm{Jianqing}\binits{J.}},
\bauthor{\bsnm{Xue},~\bfnm{Lingzhou}\binits{L.}} \AND
\bauthor{\bsnm{Zou},~\bfnm{Hui}\binits{H.}}
(\byear{2014}).
\btitle{Strong oracle optimality of folded concave
penalized estimation.}
\bjournal{Ann. Statist.}
\bvolume{42}
\bpages{819--849}.
\bid{mr={3210988}}
\end{barticle}
\bptok{imsref}%
\endbibitem

\bibitem[\protect\citeauthoryear{Fan and Yao}{2003}]{FanTimeS}
\begin{bbook}[mr]
\bauthor{\bsnm{Fan},~\bfnm{Jianqing}\binits{J.}} \AND
\bauthor{\bsnm{Yao},~\bfnm{Qiwei}\binits{Q.}}
(\byear{2003}).
\btitle{Nonlinear Time Series: Nonparametric and Parametric Methods}.
\bpublisher{Springer},
\blocation{New York}.
\bid{doi={10.1007/b97702}, mr={1964455}}
\end{bbook}
\bptok{imsref}%
\endbibitem

\bibitem[\protect\citeauthoryear{Friedman, Hastie and Tibshirani}{2008}]{graphiclasso}
\begin{barticle}[pbm]
\bauthor{\bsnm{Friedman},~\bfnm{Jerome}\binits{J.}},
\bauthor{\bsnm{Hastie},~\bfnm{Trevor}\binits{T.}} \AND
\bauthor{\bsnm{Tibshirani},~\bfnm{Robert}\binits{R.}}
(\byear{2008}).
\btitle{Sparse inverse covariance estimation with the graphical lasso}.
\bjournal{Biostatistics}
\bvolume{9}
\bpages{432--441}.
\bid{doi={10.1093/biostatistics/kxm045}, issn={1468-4357}, mid={NIHMS248717}, pii={kxm045}, pmcid={3019769}, pmid={18079126}}
\end{barticle}
\bptok{imsref}%
\endbibitem

\bibitem[\protect\citeauthoryear{Genovese et~al.}{2012}]{Genovese}
\begin{barticle}[mr]
\bauthor{\bsnm{Genovese},~\bfnm{Christopher~R.}\binits{C.~R.}},
\bauthor{\bsnm{Jin},~\bfnm{Jiashun}\binits{J.}},
\bauthor{\bsnm{Wasserman},~\bfnm{Larry}\binits{L.}} \AND
\bauthor{\bsnm{Yao},~\bfnm{Zhigang}\binits{Z.}}
(\byear{2012}).
\btitle{A comparison of the lasso and marginal regression}.
\bjournal{J. Mach. Learn. Res.}
\bvolume{13}
\bpages{2107--2143}.
\bid{issn={1532-4435}, mr={2956354}}
\end{barticle}
\bptok{imsref}%
\endbibitem

\bibitem[\protect\citeauthoryear{Harchaoui and L{\'e}vy-Leduc}{2010}]{HarLevy10}
\begin{barticle}[mr]
\bauthor{\bsnm{Harchaoui},~\bfnm{Z.}\binits{Z.}} \AND
\bauthor{\bsnm{L{\'e}vy-Leduc},~\bfnm{C.}\binits{C.}}
(\byear{2010}).
\btitle{Multiple change-point estimation with a total variation penalty}.
\bjournal{J. Amer. Statist. Assoc.}
\bvolume{105}
\bpages{1480--1493}.
\bid{doi={10.1198/jasa.2010.tm09181}, issn={0162-1459}, mr={2796565}}
\end{barticle}
\bptok{imsref}%
\endbibitem

\bibitem[\protect\citeauthoryear{Ioannidis}{2005}]{Ioannidis}
\begin{barticle}[author]
\bauthor{\bsnm{Ioannidis},~\bfnm{John~PA}\binits{J.~P.}}
(\byear{2005}).
\btitle{Why most published research findings are false}.
\bjournal{PLoS Medicine}
\bvolume{2}
\bpages{e124}.
\end{barticle}
\bptok{imsref}%
\endbibitem

\bibitem[\protect\citeauthoryear{Ising}{1925}]{Ising}
\begin{barticle}[author]
\bauthor{\bsnm{Ising},~\bfnm{Ernst}\binits{E.}}
(\byear{1925}).
\btitle{A contribution to the theory of ferromagnetism}.
\bjournal{Z. Phys}
\bvolume{31}
\bpages{253--258}.
\end{barticle}
\bptok{imsref}%
\endbibitem

\bibitem[\protect\citeauthoryear{Ji and Jin}{2012}]{UPS}
\begin{barticle}[mr]
\bauthor{\bsnm{Ji},~\bfnm{Pengsheng}\binits{P.}} \AND
\bauthor{\bsnm{Jin},~\bfnm{Jiashun}\binits{J.}}
(\byear{2012}).
\btitle{U{PS} delivers optimal phase diagram in high-dimensional variable selection}.
\bjournal{Ann. Statist.}
\bvolume{40}
\bpages{73--103}.
\bid{doi={10.1214/11-AOS947}, issn={0090-5364}, mr={3013180}}
\end{barticle}
\bptok{imsref}%
\endbibitem

\bibitem[\protect\citeauthoryear{Jin, Zhang and Zhang}{2014}]{JZ11}
\begin{barticle}[author]
\bauthor{\bsnm{Jin},~\bfnm{Jiashun}\binits{J.}},
\bauthor{\bsnm{Zhang},~\bfnm{Cun-Hui}\binits{C.-H.}} \AND
\bauthor{\bsnm{Zhang},~\bfnm{Qi}\binits{Q.}}
(\byear{2014}).
\btitle{Optimality of graphlet screening in high
dimensional variable selection}.
\bjournal{J. Mach. Learn. Res.}
\bvolume{15}
\bpages{2723--2772}.
\end{barticle}
\bptok{imsref}%
\endbibitem
%


\bibitem[\protect\citeauthoryear{Ke, Jin and Fan}{2014}]{CASEsupp}
\begin{bmisc}[author]
{\bauthor{\bsnm{Ke},~\binits{Z.~T.}},
\bauthor{\bsnm{Jin},~\binits{J.}} \AND
\bauthor{\bsnm{Fan},~\binits{J.}}}
(\byear{2014}).
\bhowpublished{Supplement to ``Covariate assisted screening
and estimation.''
DOI:\doiurl{10.1214/14-AOS1243SUPP}}.
\bptok{imsref}%
\end{bmisc}
\endbibitem

\bibitem[\protect\citeauthoryear{Lehmann and Casella}{1998}]{TPE}
\begin{bbook}[mr]
\bauthor{\bsnm{Lehmann},~\bfnm{E.~L.}\binits{E.~L.}} \AND
\bauthor{\bsnm{Casella},~\bfnm{George}\binits{G.}}
(\byear{1998}).
\btitle{Theory of Point Estimation},
\bedition{2nd} ed.
\bpublisher{Springer},
\blocation{New York}.
\bid{mr={1639875}}
\end{bbook}
\bptok{imsref}%
\endbibitem

\bibitem[\protect\citeauthoryear{Meinshausen and B{\"u}hlmann}{2006}]{Meinshausen}
\begin{barticle}[mr]
\bauthor{\bsnm{Meinshausen},~\bfnm{Nicolai}\binits{N.}} \AND
\bauthor{\bsnm{B{\"u}hlmann},~\bfnm{Peter}\binits{P.}}
(\byear{2006}).
\btitle{High-dimensional graphs and variable selection with the lasso}.
\bjournal{Ann. Statist.}
\bvolume{34}
\bpages{1436--1462}.
\bid{doi={10.1214/009053606000000281}, issn={0090-5364}, mr={2278363}}
\end{barticle}
\bptok{imsref}%
\endbibitem

\bibitem[\protect\citeauthoryear{Moulines and Soulier}{1999}]{Moulines}
\begin{barticle}[mr]
\bauthor{\bsnm{Moulines},~\bfnm{Eric}\binits{E.}} \AND
\bauthor{\bsnm{Soulier},~\bfnm{Philippe}\binits{P.}}
(\byear{1999}).
\btitle{Broadband log-periodogram regression of time series with long-range dependence}.
\bjournal{Ann. Statist.}
\bvolume{27}
\bpages{1415--1439}.
\bid{doi={10.1214/aos/1017938932}, issn={0090-5364}, mr={1740105}}
\end{barticle}
\bptok{imsref}%
\endbibitem

\bibitem[\protect\citeauthoryear{Niu and Zhang}{2012}]{SaRa}
\begin{barticle}[mr]
\bauthor{\bsnm{Niu},~\bfnm{Yue~S.}\binits{Y.~S.}} \AND
\bauthor{\bsnm{Zhang},~\bfnm{Heping}\binits{H.}}
(\byear{2012}).
\btitle{The screening and ranking algorithm to detect DNA copy number variations}.
\bjournal{Ann. Appl. Stat.}
\bvolume{6}
\bpages{1306--1326}.
\bid{doi={10.1214/12-AOAS539}, issn={1932-6157}, mr={3012531}}
\end{barticle}
\bptok{imsref}%
\endbibitem

\bibitem[\protect\citeauthoryear{Olshen et~al.}{2004}]{Olshen04}
\begin{barticle}[pbm]
\bauthor{\bsnm{Olshen},~\bfnm{Adam~B.}\binits{A.~B.}},
\bauthor{\bsnm{Venkatraman},~\bfnm{E.~S.}\binits{E.~S.}},
\bauthor{\bsnm{Lucito},~\bfnm{Robert}\binits{R.}} \AND
\bauthor{\bsnm{Wigler},~\bfnm{Michael}\binits{M.}}
(\byear{2004}).
\btitle{Circular binary segmentation for the analysis of array-based DNA copy number data}.
\bjournal{Biostatistics}
\bvolume{5}
\bpages{557--572}.
\bid{doi={10.1093/biostatistics/kxh008}, issn={1465-4644}, pii={5/4/557}, pmid={15475419}}
\end{barticle}
\bptok{imsref}%
\endbibitem

\bibitem[\protect\citeauthoryear{Ray and Tsay}{2000}]{RayTsay}
\begin{barticle}[author]
\bauthor{\bsnm{Ray},~\bfnm{Bonnie~K.}\binits{B.~K.}} \AND
\bauthor{\bsnm{Tsay},~\bfnm{Ruey~S.}\binits{R.~S.}}
(\byear{2000}).
\btitle{Long-range dependence in daily stock volatilities}.
\bjournal{J. Bus. Econom. Statist.}
\bvolume{18}
\bpages{254--262}.
\end{barticle}
\bptok{imsref}%
\endbibitem

\bibitem[\protect\citeauthoryear{Siegmund}{2011}]{Siegmundperson}
\begin{barticle}[author]
\bauthor{\bsnm{Siegmund},~\bfnm{David~O.}\binits{D.~O.}}
(\byear{2011}).
\btitle{Personal communication}.
\end{barticle}
\bptok{imsref}%
\endbibitem

\bibitem[\protect\citeauthoryear{Sun and Zhang}{2012}]{SunZhang}
\begin{barticle}[mr]
\bauthor{\bsnm{Sun},~\bfnm{Tingni}\binits{T.}} \AND
\bauthor{\bsnm{Zhang},~\bfnm{Cun-Hui}\binits{C.-H.}}
(\byear{2012}).
\btitle{Scaled sparse linear regression}.
\bjournal{Biometrika}
\bvolume{99}
\bpages{879--898}.
\bid{doi={10.1093/biomet/ass043}, issn={0006-3444}, mr={2999166}}
\end{barticle}
\bptok{imsref}%
\endbibitem

\bibitem[\protect\citeauthoryear{Tibshirani}{1996}]{Tibshirani}
\begin{barticle}[mr]
\bauthor{\bsnm{Tibshirani},~\bfnm{Robert}\binits{R.}}
(\byear{1996}).
\btitle{Regression shrinkage and selection via the lasso}.
\bjournal{J. R. Stat. Soc. Ser. B Stat. Methodol.}
\bvolume{58}
\bpages{267--288}.
\bid{issn={0035-9246}, mr={1379242}}
\end{barticle}
\bptok{imsref}%
\endbibitem

\bibitem[\protect\citeauthoryear{Tibshirani and Wang}{2008}]{TibWang}
\begin{barticle}[pbm]
\bauthor{\bsnm{Tibshirani},~\bfnm{Robert}\binits{R.}} \AND
\bauthor{\bsnm{Wang},~\bfnm{Pei}\binits{P.}}
(\byear{2008}).
\btitle{Spatial smoothing and hot spot detection for CGH data using the fused lasso}.
\bjournal{Biostatistics}
\bvolume{9}
\bpages{18--29}.
\bid{doi={10.1093/biostatistics/kxm013}, issn={1465-4644}, pii={kxm013}, pmid={17513312}}
\end{barticle}
\bptok{imsref}%
\endbibitem

\bibitem[\protect\citeauthoryear{Wasserman and Roeder}{2009}]{Wasserman}
\begin{barticle}[mr]
\bauthor{\bsnm{Wasserman},~\bfnm{Larry}\binits{L.}} \AND
\bauthor{\bsnm{Roeder},~\bfnm{Kathryn}\binits{K.}}
(\byear{2009}).
\btitle{High-dimensional variable selection}.
\bjournal{Ann. Statist.}
\bvolume{37}
\bpages{2178--2201}.
\bid{doi={10.1214/08-AOS646}, issn={0090-5364}, mr={2543689}}
\end{barticle}
\bptok{imsref}%
\endbibitem

\bibitem[\protect\citeauthoryear{Yao and Au}{1989}]{YaoAn}
\begin{barticle}[mr]
\bauthor{\bsnm{Yao},~\bfnm{Yi-Ching}\binits{Y.-C.}} \AND
\bauthor{\bsnm{Au},~\bfnm{S.~T.}\binits{S.~T.}}
(\byear{1989}).
\btitle{Least-squares estimation of a step function}.
\bjournal{Sankhy\=a Ser. A}
\bvolume{51}
\bpages{370--381}.
\bid{issn={0581-572X}, mr={1175613}}
\end{barticle}
\bptok{imsref}%
\endbibitem

\bibitem[\protect\citeauthoryear{Zhang}{2010}]{MCP}
\begin{barticle}[mr]
\bauthor{\bsnm{Zhang},~\bfnm{Cun-Hui}\binits{C.-H.}}
(\byear{2010}).
\btitle{Nearly unbiased variable selection under minimax concave penalty}.
\bjournal{Ann. Statist.}
\bvolume{38}
\bpages{894--942}.
\bid{doi={10.1214/09-AOS729}, issn={0090-5364}, mr={2604701}}
\end{barticle}
\bptok{imsref}%
\endbibitem

\bibitem[\protect\citeauthoryear{Zhang et~al.}{2010}]{ZhangSieg}
\begin{barticle}[mr]
\bauthor{\bsnm{Zhang},~\bfnm{Nancy~R.}\binits{N.~R.}},
\bauthor{\bsnm{Siegmund},~\bfnm{David~O.}\binits{D.~O.}},
\bauthor{\bsnm{Ji},~\bfnm{Hanlee}\binits{H.}} \AND
\bauthor{\bsnm{Li},~\bfnm{Jun~Z.}\binits{J.~Z.}}
(\byear{2010}).
\btitle{Detecting simultaneous changepoints in multiple sequences}.
\bjournal{Biometrika}
\bvolume{97}
\bpages{631--645}.
\bid{doi={10.1093/biomet/asq025}, issn={0006-3444}, mr={2672488}}
\end{barticle}
\bptok{imsref}%
\endbibitem

\bibitem[\protect\citeauthoryear{Zhao and Yu}{2006}]{YuB}
\begin{barticle}[mr]
\bauthor{\bsnm{Zhao},~\bfnm{Peng}\binits{P.}} \AND
\bauthor{\bsnm{Yu},~\bfnm{Bin}\binits{B.}}
(\byear{2006}).
\btitle{On model selection consistency of {L}asso}.
\bjournal{J. Mach. Learn. Res.}
\bvolume{7}
\bpages{2541--2563}.
\bid{issn={1532-4435}, mr={2274449}}
\end{barticle}
\bptok{imsref}%
\endbibitem

\bibitem[\protect\citeauthoryear{Zou}{2006}]{Zou}
\begin{barticle}[mr]
\bauthor{\bsnm{Zou},~\bfnm{Hui}\binits{H.}}
(\byear{2006}).
\btitle{The adaptive lasso and its oracle properties}.
\bjournal{J. Amer. Statist. Assoc.}
\bvolume{101}
\bpages{1418--1429}.
\bid{doi={10.1198/016214506000000735}, issn={0162-1459}, mr={2279469}}
\end{barticle}
\bptok{imsref}%
\endbibitem
\end{thebibliography}
\end{document}